\documentclass[12pt,letterpaper,twoside]{article}

\usepackage{amsmath,amssymb}
\usepackage{amsthm}
\usepackage{stmaryrd}
\usepackage{fancyhdr}
\usepackage{times}
\usepackage{mathrsfs} 
\usepackage{titlesec}
\usepackage[pdftex,dvips]{geometry}

\input diagxy

\newtheoremstyle{fact}
     {\topsep}
     {\topsep}
     {\slshape}
     {}
     {\bfseries}
     {}
     { }
     {\thmname{#1}\thmnumber{ #2.}\thmnote{ \rm (#3)}}

\newtheorem{theorem}{Theorem}[section]
\newtheorem*{theorem*}{Theorem} 
\newtheorem{lemma}[theorem]{Lemma}

\newtheorem{corollary}[theorem]{Corollary}
\newtheorem{problem}[theorem]{Problem}

\theoremstyle{definition}
\newtheorem{definition}[theorem]{Definition}
\newtheorem{remark}[theorem]{Remark}
\newtheorem{remarks}[theorem]{Remarks}
\newtheorem*{remark*}{Remark}
\newtheorem{discussion}[theorem]{Discussion}
\newtheorem{notation}[theorem]{Notation}

\newtheorem*{question*}{Question}
\newtheorem*{examples*}{Examples}  
\newtheorem{example}[theorem]{Example}
\newtheorem{examples}[theorem]{Examples}
\newtheorem*{example*}{Example}

\newtheorem*{convention*}{Convention}

\theoremstyle{fact}

\newtheorem{ftheorem}[theorem]{Theorem}
\newtheorem{flemma}[theorem]{Lemma}

\newtheorem{fcorollary}[theorem]{Corollary}

\newenvironment{myromanlist}[1][enumi]{\begin{list}{{\rm (\roman{#1})}}
{\usecounter{#1}\setlength{\labelwidth}{25pt}\setlength{\topsep}{-6pt}
\setlength{\itemsep}{-4pt} \setlength{\leftmargin}{25pt}}}{\end{list}}

\newenvironment{myalphlist}[1][enumi]{\begin{list}{{\rm (\alph{#1})}}
{\usecounter{#1}\setlength{\labelwidth}{25pt}\setlength{\topsep}{-6pt}
\setlength{\itemsep}{-4pt} \setlength{\leftmargin}{25pt}}}{\end{list}}

\def\proofont{\fontseries{bx}\fontshape{sc}\selectfont}
\def\proofname{Proof.}

\newcommand{\pcite}[2]{{\cite[#1]{#2}}}

\newcommand{\Note}[1]{}

\makeatletter
\renewenvironment{proof}[1][\proofname]{\par
  \normalfont
  \topsep6\p@\@plus6\p@ \trivlist
  \item[\hskip\labelsep\noindent\proofont #1]\ignorespaces
}{%
  \qed\endtrivlist
}
\makeatother

\titlelabel{\thetitle.\ }
\titleformat*{\section}{\normalsize\bfseries\centering}
\titleformat*{\subsection}{\normalsize\bfseries}
\titlespacing{\subsection}{0pt}{\topsep}{0.5ex}
\titleformat{\subsection}[runin]{\normalfont\bfseries}{%
\thesubsection.}{0.5ex}{}[.]

\author{W. W. Comfort\thanks{The first author gratefully acknowledges 
generous hospitality from the University of Manitoba in January 2008.} 
{ }and G\'abor Luk\'acs\thanks{The second author gratefully acknowledges 
the 
generous financial 
support received from NSERC and the University of Manitoba,
which enabled him to do this research.}}
   
\title{Locally precompact groups:\\
(Local) realcompactness and connectedness}

\begin{document}

\makeatletter
\def\@fnsymbol#1{\ifcase#1\or * \or 1 \or 2  \else\@ctrerr\fi\relax}

\let\mytitle\@title
\chead{\small\itshape W. W. Comfort and G. Luk\'acs / Locally precompact 
groups}
\fancyhead[RO,LE]{\small \thepage}
\makeatother

\maketitle

\def\thanks#1{} 

\thispagestyle{empty}

\begin{abstract}

A theorem of A.~Weil asserts that a topological group embeds as a~(dense) 
subgroup of a~locally compact group if and only if it contains a
non-empty precompact open set; such groups are called {\it locally 
precompact}. Within the class of locally precompact groups, the authors 
classify those groups with the following topological properties:

\begin{myalphlist}

\item
Dieudonn\'e completeness;

\item
local realcompactness;

\item
realcompactness;

\item
hereditary realcompactness;

\item
connectedness;

\item
local connectedness;

\item
zero-dimensionality.

\vskip 6pt
\end{myalphlist}
They also prove that an abelian locally precompact group occurs as the 
quasi-component of a topological group if and only if it is {\em 
precompactly generated}, that is, it is generated
algebraically by a precompact subset.
\end{abstract}

\setcounter{section}{-1}

\section{Introduction}

\label{sect:intro}

A subset $X$ of a topological group $G$ is {\em precompact} if for every 
neighborhood $U$ of the identity in~$G\hspace{-1pt}$, there is a finite 
\mbox{$S \hspace{-2pt} \subseteq \hspace{-2pt}X$} such that \mbox{$X 
\hspace{-2.5pt}\subseteq \hspace{-2pt} (SU) \hspace{-2.5pt}\cap 
\hspace{-2.5pt} (U \hspace{-1pt}S)\hspace{-0.5pt}$}. 
It is easily seen that every subgroup (indeed, every subset)
of a~compact group is precompact. The local 
version of that 
statement, and its converse, are the content of a~theorem of A.~Weil: A 
topological group $G$ embeds as a (dense) subgroup of a locally compact 
group $\widetilde{G}$ if and only if $G$ is {\it locally precompact} in 
the sense that some non-empty open subset of $G$ is precompact (cf. 
\cite{We3}).   For such a group $G\hspace{-1pt}$, the Weil completion 
$\widetilde{G}$ is 
unique in the obvious sense, and it coincides with the two-sided 
completion introduced by  Ra{\u{\i}}kov in 1946  (cf.~\cite{Rai}). 
For background on 
these  completions, see \cite{RoeDie}, \cite{We3}, 
\cite{Rai}, \cite[(4.11)-(4.15)]{HewRos}, and
\cite[Section 1.3]{GLCLTG}.

As the bibliographies in the monographs \cite{HewRos}, 
\cite{HofMor}, \cite{HofMor2}, and \cite{GLCLTG} attest, there 
is a huge literature devoted to the study and characterization of the locally 
compact groups that enjoy additional special topological properties. We 
work here in a parallel vein, but now in the class of locally precompact 
groups. Most of our results, when restricted back to the locally compact 
case, will be unsurprising, and in some cases familiar, to the reader.

The paper is organized as follows. After introductory material in 
\S\ref{sect:prel}, we  characterize in \S\ref{sect:real} those locally 
precompact groups that are  locally 
realcompact (Theorem~\ref{real:thm:locr}); 
they~are~the Dieudonn\'e complete groups, or equivalently,
the groups that are $G_\delta$-closed in their completions. 
In~\S\ref{sect:newHR}, we find 
internal (intrinsic) characterizations of those locally precompact groups 
that are hereditarily realcompact (Theorems~\ref{newHR:thm:general}), 
while in \S\ref{sect:conn} we address the 
relations among
connectedness properties of locally precompact groups (emphasizing the 
locally pseudocompact case) and their completions; here, the principal 
result is that a locally pseudocompact group is locally connected~if and 
only~if~its completion is locally connected 
(Theorem~\ref{conn:thm:lcpsLocCon}). \S\ref{sect:LPAq} is devoted to 
proving that within 
the class of locally precompact abelian groups, the groups $A$ that are 
topologically \mbox{isomorphic} to a~group of the form 
\mbox{$(\widetilde{G})_0\hspace{-2.5pt} \cap\hspace{-2pt} G$}
with $G$ locally pseudocompact are exactly the 
precompactly~\mbox{generated}
groups (Theorem~\ref{LPAq:thm:main});  thus, in particular, every 
connected precompact abelian group $A$ is topologically isomorphic to the 
connected component of a pseudocompact group. That theorem was established 
in~\cite[7.6]{ComfvMill2} when $A$ in addition is torsion-free, and was 
developed further in \cite[3.6]{Dikconpsc}.

\section{Definitions, notations, and preliminaries}

\label{sect:prel}

All topological spaces here are assumed to be Tychonoff.
Except when specifically noted, no algebraic assumptions are 
imposed on the groups; in particular, our groups are not necessarily 
abelian. A ``neighborhood" of a point means an {\em open} set containing 
the point. The collection of neighborhoods of the identity in a 
topological group $G$ is denoted by $\mathcal{N}(G)$. The next theorem  
explains the origin of the term precompact, and  relates it to the 
completion.

\begin{ftheorem}[\pcite{3.5}{GLCLTG}]  \label{prel:thm:charpr}
Let $G$ be a topological group, and \mbox{$X \hspace{-2.5pt}
\subseteq \hspace{-2pt}G$}~a~\mbox{subset}.
Then $X\hspace{-1.1pt}$ is precompact if and only if
$\operatorname{cl}_{\widetilde G} X$ is compact.
\end{ftheorem}

For a space $X$, we denote by $\beta X$  and $\upsilon X$ its 
Stone-\v{C}ech compactification and Hewitt realcompactification, 
respectively
(cf.~\cite[6.5,~8.4]{GilJer} and \cite[3.6.1, 3.11.16]{Engel6}).
A {\em $G_\delta$-subset}~\hspace{0.5pt}of a space $(X,\mathcal{T})$ is 
a set of the form $\bigcap\limits_{n < \omega} \hspace{-1.5pt} U_n$ with 
each $U_n \hspace{-2pt}\in \hspace{-2pt} \mathcal{T}\hspace{-1.5pt}$.
The {\em $G_\delta$-topology}  
on~$X$ is the topology generated by the $G_\delta$-subsets of 
$(X,\mathcal{T})$. A subset of $X$ is 
{\em  $G_\delta$-open} (respectively, {\em $G_\delta$-closed}, {\em 
$G_\delta$-dense}) if it is open (respectively, closed, dense) in 
the $G_\delta$-topology on $G\hspace{-1pt}$.

\begin{definition} \label{prel:def:pscp}
A space $X$ is {\em pseudocompact} if it satisfies the 
following equivalent conditions:

\begin{myromanlist}

\item
every continuous real-valued map on $X$ has bounded range;

\item
every locally finite family  of non-empty open subsets of $X$ is finite;

\item
$X$ is $G_\delta$-dense in $\beta X$.

\end{myromanlist}
\end{definition}

Definition~\ref{prel:def:pscp}(i) was introduced by Hewitt, who 
established the equivalence of (i) and (iii) (cf. \cite{HewRVCF} and 
\cite[1.4]{GilJer}). The equivalence of conditions (i) and (ii)
was shown by Glicksberg (cf.~\cite[Theorem~2]{Glick3} 
and~\mbox{\cite[3.10.22]{Engel6}}).

\begin{definition}
A topological group $G$ is said to be {\em locally pseudocompact} if there 
is \mbox{$U\hspace{-2.5pt}\in\hspace{-2pt} \mathcal{N}(G)$} such that 
$\operatorname{cl}_G U$ is pseudocompact. 
\end{definition}

Since every pseudocompact subset 
of a topological group is precompact
(cf. \cite[1.1]{ComfRoss2} and \cite[1.11]{ComfTrig}), every locally 
pseudocompact group is locally precompact.
Numerous equivalent  definitions of local pseudocompactness are provided 
in Theorem~\ref{prel:thm:lcps} below,  which summarizes the main results 
of \cite{ComfTrig}. (The paper \cite{ComfTrig} generalizes to the local 
context the results of \cite{ComfRoss2}.)

\begin{ftheorem}[\cite{ComfTrig}] \label{prel:thm:lcps}
Let $G$ be a topological group. The following statements are equivalent:

\begin{myromanlist}

\item
$G$ is locally pseudocompact;

\item
for every $V \hspace{-2.5pt}\in\hspace{-2pt} \mathcal{N}(G)$, there is 
$U\hspace{-2.5pt}\in\hspace{-2pt} \mathcal{N}(G)$ such that 
$\operatorname{cl}_G U$ is pseudocompact and 
$\operatorname{cl}_G U \hspace{-2.5pt}\subseteq\hspace{-2pt} V$;

\item
there is $U\hspace{-2.5pt}\in\hspace{-2pt} \mathcal{N}(G)$ such that
$\beta (\operatorname{cl}_G U) =\operatorname{cl}_{\widetilde G} U$;

\item
$G$ is locally precompact,  and
$\beta (\operatorname{cl}_G U) =\operatorname{cl}_{\widetilde G} U$
for every precompact $U\hspace{-2.5pt}\in\hspace{-2pt} \mathcal{N}(G)$;

\item
$G$ is locally precompact, and $\beta G = \beta \widetilde G$;

\item
$G$ is locally precompact, and $\upsilon G = \upsilon \widetilde G$;

\item
$G$ is locally precompact, and $G_\delta$-dense in $\widetilde G$.

\end{myromanlist}
\end{ftheorem}

Next, for the sake of completeness, we recall a well-known technical lemma 
concerning open subgroups of dense subgroups, which will be used several 
times in this paper. We denote by $\mathcal{H}(G)$ the set of open 
subgroups of a topological group $G\hspace{-1.25pt}$, and we set
\mbox{$o(G) \hspace{-2pt}:= \hspace{-2pt}\bigcap \mathcal{H}(G)$}. We note 
for~emphasis  that in Lemma~\ref{prel:lemma:osubgrp}, no normality 
conditions are imposed on any subgroups.

\begin{lemma} \label{prel:lemma:osubgrp}
Let $G$ be a topological group, and $D$ a dense subgroup. Then: 

\begin{myalphlist}

\item
the maps
\begin{align}
\Phi \colon \mathcal{H}(G) & \longrightarrow \mathcal{H}(D) & 
\Psi\colon \mathcal{H}(D) & \longrightarrow \mathcal{H}(G) \\
M  & \longmapsto M \cap D  &
H &\longmapsto \operatorname{cl}_G H 
\end{align}
satisfy
\mbox{$\Psi\circ\Phi =\operatorname{id}_{\mathcal{H}(G)}$} and 
\mbox{$\Phi\circ\Psi=\operatorname{id}_{\mathcal{H}(D)}$}, and thus
they are order-preserving bijections;


\item
for every $M\in \mathcal{H}(G)$, one has $|G/M| = |D/(M\cap D)|$;

\item
for every $H \in \mathcal{H}(D)$, one has
$|G/\operatorname{cl}_G H | = |D/H|$;

\item
\mbox{$o(D)\hspace{-2pt}=\hspace{-2pt}
o(G) \hspace{-2pt}\cap \hspace{-2pt} D\hspace{-1pt}$}. \qed

\end{myalphlist}

\end{lemma}

\section{Local and global realcompactness and  Dieudonn\'e completeness}

\label{sect:real}

\begin{definition} \mbox{ } \label{real:def:rc+DC}

\begin{myalphlist}

\item
A space is {\it realcompact} if it is homeomorphic to a closed subspace of 
$\mathbb{R}^\lambda$ for some cardinal $\lambda$.

\item
A space is {\it Dieudonn\'e complete} if it is homeomorphic to a closed 
subspace of a space of the form 
\mbox{$\prod\limits_{\alpha\in I}\hspace{-3pt} M_\alpha$}
with each $M_\alpha$ metrizable.

\end{myalphlist}

\end{definition}

\begin{remarks} \mbox{ } \label{real:rem:rcvsDc}

\begin{myalphlist}

\item
It is clear from Definition~\ref{real:def:rc+DC} that every realcompact
space is Dieudonn\'e complete. For limitations concerning the converse
statement, see Discussion~\ref{real:disc:Ulam} and 
Corollary~\ref{real:cor:rc=Dc?} 
below.

\item
Responding to a question posed in the fundamental memoire of
Weil~\cite{We3}, Dieudonn\'e  proved that a
space has a compatible complete uniformity if and only if it
is (in our terminology) Dieudonn\'e complete 
(cf.~\cite[p.~286]{dieu39}).  For this reason,
\pagebreak[3]
many authors prefer to call such spaces {\it topologically complete}
(cf.~\cite{GilJer} and~\cite{ComfNegB2}). The class has
been studied broadly,~for example, by Kelley 
(cf.~\cite[Chapter~15]{Kelley})
and Isbell (cf.~\cite[I.10-22]{Isb}).

\item
It is obvious from the definitions that a product of realcompact 
(respectively, Dieudonn\'e \mbox{complete}) spaces is \mbox{realcompact} 
(respectively, Dieudonn\'e complete), and that a closed subspace of a 
realcompact (respectively, Dieudonn\'e complete) space is realcompact 
(respectively, Dieudonn\'e complete). Since the intersection 
\mbox{$\bigcap\limits_{\alpha\in I} \hspace{-3pt} A_\alpha$} of subspaces 
of any (fixed) space is homeomorphic to a closed \mbox{subspace}~of
\mbox{$\prod\limits_{\alpha\in I} \hspace{-3pt} A_\alpha$},
it is further immediate from the
definitions that in any space $Y\hspace{-2pt}$, each
subspace of the form 
\mbox{$\bigcap\limits_{\alpha\in I} \hspace{-3pt} A_\alpha$}
with each $A_\alpha$ a~realcompact
(respectively, Dieudonn\'e complete) subspace of $Y$
is itself realcompact (respectively, Dieudonn\'e complete).

\vspace{6pt}

\end{myalphlist}
The statements given in the next theorems, which are all basic in the
study of realcompact spaces and of Dieudonn\'e complete spaces, are less
obvious; we will rely on these properties in what follows.
\end{remarks}

\begin{ftheorem} \label{real:thm:bas}
\mbox{ }

\begin{myalphlist}

\item
{\rm (\cite[2.3]{ComfHernTrig})}
Every $G_\delta$-closed subspace of a realcompact
{\rm (\!}respectively Dieudonn\'e complete{\rm )} space is 
realcompact {\rm (\!}respectively, Dieudonn\'e complete{\rm )}.

\item
{\rm (\cite[8.2]{GilJer}, \cite[3.11.12]{Engel6})}
Every Lindel\"of space is realcompact.

\item
{\rm (\cite[3.6]{ComfHernTrig})}
Every locally compact topological group is Dieudonn\'e complete.

\end{myalphlist}
\end{ftheorem}

\begin{notation}\label{real:not:ups+gamma}
With each space $X$ are associated spaces $\upsilon X$ and $\gamma X$
defined as follows:

\begin{myalphlist}

\item
\mbox{$\upsilon X \hspace{-1.7pt}:= \hspace{-1pt}
\{p \hspace{-2pt}\in \hspace{-2pt}\beta X \mid \text{each continuous map 
from $X$ to $\mathbb R$ extends continuously to } p\}$;}

\item
\mbox{$\gamma X \hspace{-1.7pt}:= \hspace{-1pt}
\{p\hspace{-2pt}\in\hspace{-2pt} \beta X\mid \text{each 
continuous map from $X$ to a metric space extends continuously to } p\}$.}

\end{myalphlist}
\end{notation}

\begin{ftheorem}[\pcite{Chapter~8}{GilJer}, \pcite{3.11.16, 
8.5.13}{Engel6},
and \pcite{pp.~1-20}{ComfNegB2}]
\label{real:thm:rcpt-tc}
Let $X$ be a space. Then:

\begin{myalphlist}

\item
$\upsilon X$ is realcompact and
\mbox{$\upsilon X\hspace{-2.5pt} =\hspace{-1.25pt}
\bigcap\{X^\prime \mid 
X \hspace{-2pt}\subseteq \hspace{-2pt}  X^\prime
 \hspace{-2pt} \subseteq \hspace{-2pt} \beta X
,\ X^\prime \text{ is realcompact}\}$}; 

\item
$\gamma X$ is Dieudonn\'e complete and
\mbox{$\gamma X\hspace{-2.5pt} =\hspace{-1.25pt}
\bigcap\{X^\prime \mid 
X \hspace{-2pt}\subseteq \hspace{-2pt}  X^\prime
 \hspace{-2pt} \subseteq \hspace{-2pt} \beta X
,\ X^\prime \text{ is Dieudonn\'e complete}\}$}.

\end{myalphlist}
\end{ftheorem}

Realcompact spaces (under the name of {\em $Q$-spaces}) 
as well as the space $\upsilon X$ were introduced~by Hewitt
(cf.~\cite[Definition~12, Theorems~56-60]{HewRVCF}).
Accordingly, the space $\upsilon X$ is called the {\it 
Hewitt realcompactification} of $X$. Similarly, honoring 
Dieudonn\'e, the space $\gamma X$ is called the {\it Dieudonn\'e 
completion} of $X$ (cf.~\cite{dieu39}).

The definitions of realcompactness and Dieudonn\'e completeness 
are similar, yet different.~The distinction is best described by using the 
set-theoretic notion of 
Ulam-measurable cardinals. 
A~cardinal number $\lambda$ is said to be {\em Ulam-measurable} if 
there is a~non-atomic countably additive \mbox{measure}
\mbox{$\mu\colon\mathcal{P}(\lambda)\rightarrow\{0,1\}$} such that 
\mbox{$\mu(\lambda)\hspace{-2pt} = \hspace{-1pt}1$.}
Ulam-measurable 
cardinals are called {\it measurable} in the text~\cite{GilJer}, 
but we follow standard procedure in reserving that term 
for cardinals $\lambda$ with a measure 
\mbox{$\mu\colon\mathcal{P}(\lambda)\twoheadrightarrow\{0,1\}$}
that is \mbox{$<\hspace{-2.5pt}\lambda$}-additive 
in the sense that every 
\mbox{$A\hspace{-2pt}\subseteq\hspace{-2pt}\lambda$} 
with \mbox{$|A|\hspace{-2pt}<\hspace{-2pt}\lambda$}
satisfies \mbox{$\mu(A)\hspace{-2.25pt}=\hspace{-1pt}0\hspace{-0.5pt}$.}

\begin{discussion} \label{real:disc:Ulam}
The existence of Ulam-measurable cardinals cannot be proven in ZFC---that 
is, their non-existence is consistent with the axioms of ZFC 
(cf.~\cite[IV.6.9, VI.4.13]{KunenST}). 
Most set theorists (appear to) believe that the existence of an
Ulam-measurable cardinal is consistent with the axioms of ZFC, but that 
has not 
been established (cf.~\cite{KunenST}).
It is known that an Ulam-measurable cardinal exists if and only 
if an uncountable measurable cardinal exists. Indeed, the first 
Ulam-measurable cardinal $\mathfrak{m}$ (if it exists)
is measurable (cf.~\cite{Ulam}, \cite{Tarski}, and 
\cite[12.5(ii)]{GilJer}). 
\pagebreak[3]
Henceforth, we write 
\mbox{$\lambda \hspace{-2pt}<\hspace{-2.25pt} \mathfrak{m}$} instead of
``$\lambda$ is not Ulam-measurable." Such statements are to be read with 
some good will: If no Ulam-measurable cardinal exists, then the expression
$\lambda<\mathfrak{m}$
is vacuously true for every cardinal $\lambda$.
\end{discussion}

The relevance of Ulam-measurable cardinals to our work is given by the 
following consequence of a theorem of Mackey (cf.~\cite{MackeyMeas}).

\begin{ftheorem}[\pcite{12.2}{GilJer}]
\label{real:thm:discrete}
A discrete space $D$ is realcompact if and only if 
\mbox{$|D|\hspace{-2pt}<\hspace{-2.25pt}\mathfrak{m}$}.
\end{ftheorem}

Recall that a {\em cellular family} in a space $X$ is a collection of
non-empty, pairwise disjoint open subsets of $X$. The {\em cellularity} of 
$X$ is defined by the relation
\begin{align}
c(X) := 
\sup\{|\mathcal{U}| : \mathcal{U} \text{ is a cellular family in } X\}.
\end{align}
The following consequence of a theorem of  Shirota (cf.~\cite{Shirota})
provides a sufficient condition for Dieudonn\'e complete spaces 
to be realcompact.

\begin{ftheorem}[\pcite{15.20}{GilJer}] \label{real:thm:cell}
If $X$ is Dieudonn\'e complete and 
\mbox{$c(X) \hspace{-2pt} < \hspace{-2.25pt} \mathfrak{m}$}, then
$X$ is realcompact.
\end{ftheorem}

It is easily seen that a metrizable space of
Ulam-measurable cardinality contains a closed, discrete subspace of
Ulam-measurable cardinality (cf.~\cite[6.2]{ComfNegB2}). Thus, 
combining Theorems~\ref{real:thm:discrete} and~\ref{real:thm:cell} 
yields the following useful result.

\begin{fcorollary}[\cite{GilJer}, \cite{ComfNegB2}]\label{real:cor:rc=Dc?}
The following statements are equivalent:

\begin{myromanlist}

\item
there is no Ulam-measurable cardinal;

\item
the class of realcompact spaces coincides with the class of
Dieudonn\'e complete spaces.

\end{myromanlist}
\end{fcorollary}

It is well known that a space is compact if and only if it is 
pseudocompact and realcompact (cf.~\cite[Theorem~54]{HewRVCF}
and~\cite[3.11.1]{Engel6}). Thus, the 
notion of realcompactness is a natural complement to that of 
pseudocompactness. While Theorem~\ref{prel:thm:lcps} provides a complete 
``internal" characterization of locally pseudocompact groups,  we are 
aware of no parallel intrinsic characterization of (locally) realcompact
groups. In this section, we remedy this  deficiency for locally precompact 
groups.  Since every Lindel\"of space is realcompact 
(cf.~\cite[8.2]{GilJer} and~\cite[3.11.12]{Engel6}),
a complete description of realcompact groups is beyond the scope of this 
paper. Our approach is based on an argument that was used in 
\cite[Section~4]{ComfHernTrig}, which we formulate here explicitly.

For a topological space $X$, a {\em zero-set in $X$} is a set of the form 
\mbox{$f^{-1}(0)$,} 
where $f$ is a~real-valued continuous function on 
$X\hspace{-1.5pt}$. A subset \mbox{$Y \hspace{-2pt}\subseteq\hspace{-2pt} 
X$}
is {\em $z$-embedded} in $X$ if for every zero-set 
$Z$~in~$Y\hspace{-1.5pt}$, there is a zero-set $W$ in $X$ such that 
\mbox{$Z\hspace{-2pt}= \hspace{-2pt} W\hspace{-3pt}\cap 
\hspace{-2ptY}\hspace{-2.5pt}$.}
(To our best knowledge, this concept was first introduced into the 
literature by Isbell \cite{Isbell5}, and explicitly by Henriksen and 
Johnson~\cite{HenJohn}; see also Hager~\cite{Hager1} for additional 
citations and applications.)
One says that $X$ is an {\em $Oz$-space} if every open subset of $X$ is 
$z$-embedded (cf.~\cite{Blair}). 
Recall that a subset \mbox{$F\hspace{-2pt} \subseteq 
\hspace{-2pt}X$}
is {\em regular-closed} if 
\mbox{$F\hspace{-2pt}= \hspace{-1pt}  \operatorname{cl}_X 
(\operatorname{int}_X \hspace{-1.5pt}F)$,} or equivalently, if
\mbox{$F \hspace{-2.pt}=\hspace{-1pt} \operatorname{cl}_X \hspace{-1pt}U$}
for an open subset 
\mbox{$U \hspace{-2pt}\subseteq \hspace{-2pt}X\hspace{-1.5pt}$.}
Blair has characterized $Oz$-spaces in several ways.

\begin{ftheorem}[\pcite{5.1}{Blair}] \label{real:thm:Oz}
For every space $X$, the following statements are equivalent:

\begin{myromanlist}

\item
$X$ is an $Oz$-space;

\item
every dense subset of $X$ is $z$-embedded in $X\hspace{-1.5pt}$;

\item
every regular-closed subset of $X$ is a zero-set in $X\hspace{-1.75pt}$.

\end{myromanlist}
\end{ftheorem}

\pagebreak[3]

\begin{ftheorem}

\label{real:thm:BlaHag}

\mbox{ }

\begin{myalphlist}





\item
{\rm (\cite[5.3]{Blair})}
If $X$ is an $Oz$-space and \mbox{$S\hspace{-2pt}\subseteq\hspace{-2pt}X$}
is dense or open or regular-closed, then $S$ is an $Oz$-space.

\item
{\rm (\cite[1.1(b)]{BlaHag2})}
If\mbox{ $\hspace{1.75pt}Y\hspace{-1.75pt}$ }is $z$-embedded in 
\mbox{$X\hspace{-1.5pt}$,}
then $\upsilon Y$ is the $G_\delta$-closure of\ \ $Y$ in 
$\upsilon X\hspace{-1.25pt}$; hence, 
\mbox{$\upsilon Y \hspace{-2.6pt}\subseteq\hspace{-2pt} 
\upsilon X\hspace{-2pt}$}.

\end{myalphlist}
\end{ftheorem}

A key component of our treatment of locally compact groups 
(and their subgroups) is the following consequence of a result 
of Ross and Stromberg:

\begin{ftheorem}[\pcite{1.3, 1.6}{RossStro}, \pcite{1.10}{ComfTrig}] 
\label{real:thm:RS}
Every locally compact group is an $Oz$-space.
\end{ftheorem}

Since every locally precompact group is a dense subgroup of a locally 
compact group, Theorems~\ref{real:thm:RS} 
and~\ref{real:thm:BlaHag}(a) yield:

\begin{fcorollary}[\pcite{1.10}{ComfTrig}] \label{real:cor:LPOz}
Every locally precompact group is an $Oz$-space. 
\end{fcorollary}

\begin{lemma} \label{real:lemma:zups}
Let $G$ be a locally precompact group, and
\mbox{$U \hspace{-2pt}\subseteq\hspace{-2pt} G$} an open subset.
Put \mbox{$F\hspace{-2.5pt} := \hspace{-1.25pt}
\operatorname{cl}_G U\hspace{-2pt}$} and
\mbox{$K\hspace{-2.5pt}:= \hspace{-1.25pt}
\operatorname{cl}_{\widetilde G} U\hspace{-2pt}$.}
Then $F$ is $z$-embedded in $K\hspace{-1pt}$, and $\upsilon F$ is the 
$G_\delta$-closure of $F$ in $\upsilon K\hspace{-1pt}$.
\end{lemma}

\begin{proof}
Since $U$ is open in $G$,
there is an open subset \mbox{$V \hspace{-2.25pt}\subseteq \hspace{-2pt} 
\widetilde{G}$} such that 
\mbox{$U\hspace{-2.75pt}= 
\hspace{-1.75pt} V\hspace{-2.75pt} \cap \hspace{-1.9pt} G$\hspace{-1pt}}. 
Thus,~one~has 
\mbox{$K  \hspace{-2.25pt}= \hspace{-1pt} \operatorname{cl}_{\widetilde G} U 
\hspace{-2pt} =  \hspace{-1pt}
\operatorname{cl}_{\widetilde G} (V \hspace{-2.75pt}\cap \hspace{-1.9pt} G) 
\hspace{-2pt} =  \hspace{-1pt}
\operatorname{cl}_{\widetilde G} V\hspace{-2pt}$}, 
because $G$ is dense in $\widetilde G\hspace{-1pt}$. Therefore,
$K$ is a regular-closed subset of the $Oz$-space
$\widetilde G\hspace{-1pt}$ (cf. Theorem~\ref{real:thm:RS}), and
by Theorem~\ref{real:thm:BlaHag}(a), $K$ is an $Oz$-space. 
So, by 
Theorem~\ref{real:thm:Oz}, every dense subset of $K$ is $z$-embedded in 
$K\hspace{-1.5pt}$. In 
particular, $F$ is $z$-embedded in $K\hspace{-1.5pt}$. Hence, by 
Theorem~\ref{real:thm:BlaHag}(b),
$\upsilon F$ is the $G_\delta$-closure of $F$ in 
$\upsilon K\hspace{-1pt}$.
\end{proof}

\begin{remark} 
In developing our proof of Lemma~\ref{real:lemma:zups}, we have followed 
the authors of \cite[2.3]{ComfTrig} in relying on the results cited from  
\cite{RossStro}, \cite{BlaHag2}, and \cite{Blair}. We  note that 
alternative sources for equivalent statements are available in 
the literature: The fact that every locally compact group is (in our 
terminology) an $Oz$-space follows immediately from \v{S}\v{c}epin's 
results (cf.~\cite{Scepin1}~and \cite{Scepin2}); 
Tkachenko has shown that every $G_\delta$-dense subspace of an $Oz$-space 
is $C$-embedded (cf.~\cite[Theorem~2]{TkaCemb}).
\end{remark}

A topological space $X$ is said to be {\em locally realcompact} 
(respectively, {\em locally Dieudonn\'e complete}) if for every 
\mbox{$x\hspace{-2pt}\in\hspace{-2pt} X\hspace{-1.5pt}$},  
there is a neighborhood $U$ of 
$x$  such that  \mbox{$\operatorname{cl}_X \hspace{-1pt} U$} is 
realcompact (respectively, Dieudonn\'e complete). Since our 
spaces are Tychonoff, and the properties in question are inherited by
closed subspaces, it is clear that a space $X$  is locally 
realcompact (respectively, locally Dieudonn\'e complete)
if and only if for each \mbox{$x\hspace{-2pt}\in\hspace{-2pt} X$}
and neighborhood $U$ of $x$ there is a neighborhood $V$ of $x$
such that \mbox{$\operatorname{cl}_X \hspace{-1.5pt}V$} is realcompact 
(respectively, Dieudonn\'e complete) and
\mbox{$\operatorname{cl}_X \hspace{-1.5pt} V \hspace{-2pt}
\subseteq\hspace{-2pt} U$\hspace{-2pt}}. Echoing the relationship 
between a locally compact space and its Stone-\v{C}ech compactification, 
a~space $X$ is locally realcompact (respectively, locally Dieudonn\'e 
complete) if and 
only if $X$ is open in its Hewitt realcompactification $\upsilon X$ 
(respectively, in its Dieudonn\'e completion $\gamma X$) 
(cf.~\cite[2.11]{MaRaWo}).

In order to characterize global and local  realcompactness and Dieudonn\'e 
completeness in the class of locally precom\-pact groups, one introduces 
a cardinal invariant. 

\begin{definition}
Let $\tau$ be an infinite cardinal, and $G$ a topological group.

\begin{myalphlist}

\item
A subset $X$ of $G$ is said to be {\em $\tau$\mbox{-}precompact}
if for every 
\mbox{$U\hspace{-2.5pt} \in \hspace{-1.5pt}\mathcal{N}(G)$,} there 
is \mbox{$S  \hspace{-2pt} \subseteq  \hspace{-2pt} X$} that satisfies
$|S|  \hspace{-2pt} \leq\nolinebreak  \hspace{-2pt}\tau$~and
 \mbox{$X \hspace{-2.5pt}\subseteq \hspace{-2pt} (SU) 
\hspace{-2.5pt}\cap \hspace{-2.5pt} (U \hspace{-1pt}S)\hspace{-1pt}$}
(cf. ``$\tau$-bounded" in \cite{Guran}). 

\item
The {\em precompactness index} \ 
$ip(X)$ of a subset $X$ of $G$ is the least infinite cardinal $\tau$ such 
that $X$ is 
$\tau$\mbox{-}precompact (cf. ``index of boundedness" in \cite{TkaIntro}). 
\end{myalphlist}
\end{definition}

\pagebreak[3]

\begin{remark}
We note that the precompactness index is not a topological invariant of a 
space~$X\hspace{-2pt}$, but rather of the way a space $X$ is placed in 
$G\hspace{-1pt}$. Indeed, homeomorphic subspaces of a given group $G$ may 
have different 
precompactness indices, as the following example shows:
Let \mbox{$\lambda \hspace{-2pt} >\hspace{-2pt} \omega$} be a~cardinal,
$E$ a discrete group of cardinality $\lambda$, put
\mbox{$G \hspace{-2pt} := \hspace{-2pt}
(\mathbb{Z}/2\mathbb{Z})^\lambda \hspace{-2.75pt} \times \hspace{-2pt} E$}, 
and let $D$ be a discrete subset of cardinality $\lambda$ of 
$(\mathbb{Z}/2\mathbb{Z})^\lambda\hspace{-2pt}$. 
(For instance, one can take $D$ to be 
the set of elements with precisely one non-zero coordinate.) Then 
\mbox{$D\hspace{-2pt} \times \hspace{-2.5pt} \{e\}$} and 
\mbox{$\{e\} \hspace{-2.5pt} \times \hspace{-2.25pt} E$} 
are \mbox{homeomorphic},~but $D$ is precompact, 
and so \mbox{$ip(D\hspace{-2pt} \times  \hspace{-2.5pt} \{e\})
\hspace{-2pt}=\hspace{-1pt}\omega\hspace{-0.5pt}$}, 
while \mbox{$ip(\{e\} \hspace{-2.5pt} \times \hspace{-2.25pt} E)
\hspace{-2pt}=\hspace{-1pt}\lambda$}.
Nevertheless, 
if $H$ is a subgroup of $G$ that contains $X\hspace{-2pt}$, then $X$ has 
the same precompactness index in $H$ and in $G$ (cf.~\cite[2.24(d)]{GLCLTG}). 
\end{remark}

In what follows, we need the following elementary properties of the
precompactness index. (Theorem~\ref{real:thm:LCLindelof} below has an 
obvious analogue for  cardinals 
\mbox{$\lambda\hspace{-2pt}\geq\hspace{-2pt}\omega$}, but we require only 
the case \mbox{$\lambda\hspace{-2pt}=\hspace{-2pt}\omega$}.)

\begin{ftheorem}[\pcite{1.29}{GLCLTG}] \label{real:thm:LCLindelof}
For every locally compact group $L$, the following statements are 
equivalent:

\begin{myromanlist}

\item
$L$ is $\omega$-precompact;

\item
$L$ is $\sigma$-compact;

\item
$L$ is Lindel\"of.

\end{myromanlist}
\end{ftheorem}

\begin{ftheorem} \label{real:thm:ip}
Let $G$ be a topological group, and $X$ a subset of $G$. 

\begin{myalphlist}

\item
{\rm (\cite[2.24(a)]{GLCLTG})}
If \mbox{$Y \hspace{-2pt}\subseteq \hspace{-2pt} X$\hspace{-1.5pt}},
then \mbox{$ip(Y)\hspace{-2pt} \leq  \hspace{-1.5pt} ip(X)$.}

\item
{\rm (\cite[3.2]{DikTkaWCFTP}, \cite[2.24(c)]{GLCLTG})}
\mbox{$ip(\operatorname{cl}_G \hspace{-1pt} X) 
\hspace{-2pt} = \hspace{-1.5pt} ip(X)$.}

\item
{\rm (\cite[2.30]{GLCLTG})}
\mbox{$ip(\langle X\rangle) \hspace{-2pt} = \hspace{-1.5pt} ip(X)$.}

\end{myalphlist}
\end{ftheorem}

In \cite{ComfHernTrig}, the authors used the {\em compact covering number} 
$\kappa(X)$ (i.e., the smallest number of compact subsets of $X$
that cover $X$) to characterize realcompactness in the context of locally 
compact groups.  It is easily seen that 
\mbox{$ip(L)\hspace{-2pt} = \hspace{-2pt} \omega \cdot \kappa(L)$} 
for every infinite locally compact group~$L$.
Theorem~\ref{real:thm:ip}(b) indicates that for locally precompact 
groups, the precompactness index is the correct cardinal invariant to 
consider.

\begin{theorem} \label{real:thm:c-ip}
Let $G$ be a locally precompact group, and 
\mbox{$U \hspace{-2.5pt} \subseteq\hspace{-2pt} G$} an open subset. 
Then:

\begin{myalphlist}

\item
\mbox{$c(U) \hspace{-2pt} \leq \hspace{-1.5pt} ip(U)$};

\item
if \mbox{$\operatorname{cl}_G \hspace{-1pt} U$}
is Dieudonn\'e complete and 
\mbox{$ip(U) \hspace{-2pt} < \hspace{-2pt} \mathfrak{m}$}, then
\mbox{$\operatorname{cl}_G \hspace{-1pt} U$} is realcompact.

\end{myalphlist}
\end{theorem}

\begin{proof}
(a) Since $ip(U)$ is independent of the ambient group $G$, by replacing 
the group
$G$ with the subgroup $\langle U \rangle$ generated by $U$ if necessary, 
we may assume that 
\mbox{$G\hspace{-2pt} = \hspace{-2pt} \langle U \rangle$}. Thus, 
by Theorem~\ref{real:thm:ip}, 
\mbox{$ip(\widetilde G) \hspace{-2pt} =  \hspace{-1.5pt} ip(G) 
\hspace{-2pt} = \hspace{-1.5pt} ip(U)$}. Since $G$ is locally precompact, 
its completion $\widetilde G$ is locally compact, and so
\mbox{$ip(\widetilde G)\hspace{-2pt} = \hspace{-1.25pt}
\omega \cdot \kappa(\widetilde G)$}. Therefore, by a theorem of Tkachenko,
\mbox{$c(\widetilde G) \hspace{-2pt} \leq \hspace{-1.25pt} 
\omega \cdot \kappa(\widetilde G) \hspace{-2pt} = \hspace{-1.5pt}
ip(\widetilde G)$} (cf.~\cite[4.8]{TkaIntro}). Hence, 
\begin{align}
c(U) \leq c(G) = c(\widetilde G)  \leq ip(\widetilde G) = ip (U).
\end{align}

(b) By (a),
\mbox{$c(\operatorname{cl}_G \hspace{-1pt} U) \hspace{-2pt} 
= \hspace{-1.5pt} c(U) \hspace{-2pt} \leq \hspace{-1.5pt} ip(U) <
\hspace{-2pt} \mathfrak{m}$}. Thus, the statement follows by 
Theorem~\ref{real:thm:cell}.
\end{proof}

\begin{remark}
The hypothesis in Theorem \ref{real:thm:c-ip}(a) that $G$ is
locally precompact cannot be omitted: Indeed, put
\mbox{$G \hspace{-1pt}:= \bigoplus\limits_{\omega_1} 
\mathbb{Z}/2\mathbb{Z}$}, and  equip~$G$~with the group topology whose 
base at zero consists of subgroups
\mbox{$H_\alpha:=\{x \hspace{-2pt}\in G\hspace{-2pt} \mid 
x_\beta \hspace{-2pt}= \hspace{-2pt}0 \mbox{ for all } 
\beta \hspace{-2pt} < \hspace{-2pt}\alpha\}$}, where 
\mbox{$\alpha  \hspace{-2pt} < \hspace{-2pt}\omega_1$}.
Since the quotient \mbox{$G/H_\alpha$} is countable for every 
\mbox{$\alpha \hspace{-2pt} < \hspace{-2pt}\omega_1$},
it follows that \mbox{$ip(G)\hspace{-1.5pt}=\hspace{-1pt}\omega$}, and 
thus, by Theorem~\ref{real:thm:ip}(a),
\mbox{$ip(U)\hspace{-1.5pt}=\hspace{-1pt}\omega$} for every open subset 
$U$ of $G\hspace{-1pt}$. On the other hand, if 
\pagebreak[3]
\mbox{$e^{(\gamma)}\hspace{-2pt}\in\hspace{-2pt} G$} is such that
 \mbox{$e^{(\gamma)}_\beta \hspace{-2pt} = \hspace{-2pt} 1$}
if and only if \mbox{$\gamma \hspace{-2pt} = \hspace{-2pt}\beta$},
then \mbox{$\{H_\gamma + e^{(\gamma)}\}_{\alpha \leq \gamma < \omega_1}$}
is a pairwise disjoint family of open subsets of $H_\alpha$. Therefore,
\mbox{$c(H_\alpha)\hspace{-2pt} = \hspace{-2pt}\omega_1$} for every
\mbox{$\alpha \hspace{-2pt} < \hspace{-2pt}\omega_1$}, and hence
\mbox{$ip(U)\hspace{-2pt}<\hspace{-2pt}  c(U)\hspace{-2pt} = 
\hspace{-2pt}\omega_1$} for every non-empty 
open subset of $G\hspace{-1pt}$. (The group $G$ was defined 
and considered in \cite[3.2]{ComfRoss2} for a different, but related, 
purpose.)
\end{remark}

We now turn to identifying locally realcompact groups within the class 
of locally precompact groups. Unexpectedly, these prove to be exactly the
(locally) Dieudonn\'e complete groups in the class. Therefore,
Theorem~\ref{real:thm:locr} below  provides a positive answer to a 
special case of a problem of Arhangel{$'$}ski\u{\i} and Tkachenko 
(cf.~\cite[3.2.2]{ArhTka}), who asked whether every locally Dieudonn\'e 
complete topological group is Dieudonn\'e complete.

\begin{theorem} \label{real:thm:locr}
Let $G$ be a locally precompact group. The following statements are 
equivalent:

\begin{myromanlist}

\item
$G$ is Dieudonn\'e complete;

\item
$G$ is locally Dieudonn\'e complete;

\item
$G$ is locally realcompact;

\item
$G$ is $G_\delta$-closed in $\widetilde G$;

\item
every open subgroup of $G$ is $G_\delta$-closed in $\widetilde G$;

\item
$G$ contains an open subgroup that is $G_\delta$-closed in $\widetilde G$;

\item
every $\omega$-precompact open subgroup of $G$ is realcompact;

\item
$G$ contains a realcompact open subgroup;

\item
$G$ contains a Dieudonn\'e complete open subgroup.

\end{myromanlist}
\end{theorem}

\begin{proof}
The logical scheme of the proof is as follows:
\begin{align}
\bfig
\morphism(0,0)/=>/<450,0>[\textrm{(vi)}`\textrm{(iv)};]
\morphism(0,0)/<=/<0,-330>[\textrm{(vi)}`\textrm{(v)};]
\morphism(0,-330)/<=/<450,330>[\textrm{(v)}`\textrm{(iv)};]
\morphism(450,0)/=>/<450,0>[\textrm{(iv)}`\textrm{(i)};]
\morphism(900,0)/=>/<450,0>[\textrm{(i)}`\textrm{(vii)};]
\morphism(450,0)/<=/<0,-330>[\textrm{(iv)}`\textrm{(iii)};]
\morphism(1350,0)/=>/<450,-330>[\textrm{(vii)}`\textrm{(viii)};]
\morphism(450,-330)/<=/<450,0>[\textrm{(iii)}`\textrm{(ii)};]
\morphism(1350,-330)/<=/<450,0>[\textrm{(ix)}`\textrm{(viii)};]
\morphism(900,-330)/<=/<450,0>[\textrm{(ii)}`\textrm{(ix)};]
\efig
\end{align}

The implications (iv) $\Rightarrow$ (v) $\Rightarrow$ (vi),
and (viii)  $\Rightarrow$ (ix) are obvious.

(vi) $\Rightarrow$ (iv): Let $H$ be an open subgroup of $G$ that is 
$G_\delta$-closed in $\widetilde G\hspace{-1pt}$. By 
Lemma~\ref{prel:lemma:osubgrp}(a), 
\mbox{$M\hspace{-2.25pt}: = \hspace{-1pt} 
\operatorname{cl}_{\widetilde G}\hspace{-1pt}H$} 
is an open subgroup of $\widetilde G\hspace{-1pt}$.
Thus, $G$ and $\widetilde G$ are homeomorphic  (as topological spaces)
to \mbox{$H \hspace{-2.5pt} \times \hspace{-1.5pt} (G 
\hspace{-0.5pt}/\hspace{-1pt} H)$} and 
 \mbox{$M \hspace{-2.5pt} \times \hspace{-1.5pt} (\widetilde G 
\hspace{-0.5pt}/\hspace{-1pt} M)$}, respectively,
where both \mbox{$ G\hspace{-0.5pt}/\hspace{-1pt} H$} and 
\mbox{$ \widetilde G\hspace{-0.5pt}/\hspace{-1pt} M$} are discrete
(cf.~\cite[5.26]{HewRos}).
By Lemma~\ref{prel:lemma:osubgrp}(c),~one~has
\mbox{$|G/H| \hspace{-2pt} = \hspace{-2pt} |\widetilde G/M|$}. 
Therefore, we obtain the following commutative diagram with the 
horizontal arrows representing homeomorphisms (as topological spaces):
\begin{align} \label{real:diag:GM}
\bfig
\square|alra|/->```->/<1000,450>%
[G`H \times (G/H)`\widetilde G`M \times (\widetilde G/M);\sim```\sim]
\morphism(0,370)/^{(}->/<0,-370>[`\widetilde G;]
\morphism(1000,370)/^{(}->/<0,-370>[`M \times (\widetilde G/M);]
\efig
\end{align}
Hence, the statement follows from the fact that
$H$ is $G_\delta$-closed~in~$M\hspace{-2pt}$.

(iv) $\Rightarrow$ (i): 
Since $G$ is locally precompact, $\widetilde G$ is locally 
compact, and by Theorem~\ref{real:thm:bas}(c), $\widetilde G$ is
Dieudonn\'e complete.
Thus, by Theorem~\ref{real:thm:bas}(a), $G$ is
Dieudonn\'e complete, being $G_\delta$-closed 
in~$\widetilde{G}\hspace{-1pt}$.

(i) $\Rightarrow$ (vii): Let $H$ be an $\omega$-precompact open subgroup 
of $G\hspace{-1pt}$. Then $H$ is closed in $G$, and so~by 
Remark~\ref{real:rem:rcvsDc}(c),~$H$~is Dieudonn\'e complete.
As \mbox{$ip(H) \hspace{-2pt} \leq \hspace{-1.5pt} \omega$},
by Theorem~\ref{real:thm:c-ip}(b),  $H$ is realcompact.

\pagebreak[3]

(vii) $\Rightarrow$ (viii): Since $G$ is locally precompact, there exists
\mbox{$U \hspace{-2.5pt} \in \hspace{-2pt} \mathcal{N}(G)$} such that $U$ 
is  \mbox{precompact}. Then
\mbox{$ip(U)\hspace{-2pt} \leq\hspace{-2pt} \omega$}, and so by
Theorem~\ref{real:thm:ip}(c), one has
\mbox{$ip(\langle U \rangle)  \hspace{-2pt} =\hspace{-1.5pt}
ip(U)\hspace{-2pt} \leq\hspace{-2pt} \omega$}. Therefore,
\mbox{$H \hspace{-2pt} :=\hspace{-2pt} \langle U \rangle$} is an 
$\omega$-precompact open subgroup of $G$. Hence, by (vii), $H$ is
realcompact.

(ix) $\Rightarrow$ (ii): 
Let $H$ be a Dieudonn\'e complete open subgroup of $G\hspace{-1pt}$.
Then $H$ is closed,~and~thus $G$ is locally Dieudonn\'e complete.

(ii) $\Rightarrow$ (iii): Let 
\mbox{$U \hspace{-2.5pt} \in \hspace{-2pt} \mathcal{N}(G)$} be such that
\mbox{$\operatorname{cl}_G \hspace{-1pt} U$} is Dieudonn\'e complete.
Since $G$ is locally precompact, there is 
\mbox{$V \hspace{-3pt} \in \hspace{-2pt} \mathcal{N}(G)$} such that
$V$ is precompact. Put 
\mbox{$W\hspace{-2.5pt}:= \hspace{-2pt} U \hspace{-3pt} \cap 
\hspace{-2.5pt}V\hspace{-2.5pt}$}. By Theorem~\ref{real:thm:ip}(a), one 
has \mbox{$ip(W) \hspace{-2pt}\leq
\hspace{-1.5pt} ip(V) \hspace{-2pt} \leq \hspace{-1.25pt} \omega 
\hspace{-2pt} <  \hspace{-2pt} \mathfrak{m}$}. By 
Remark~\ref{real:rem:rcvsDc}(c), 
\mbox{$\operatorname{cl}_G \hspace{-1pt} W$} is Dieudonn\'e complete, 
being a closed subspace of 
\mbox{$\operatorname{cl}_G \hspace{-1pt} U\hspace{-2pt}$}.
Therefore, by Theorem~\ref{real:thm:c-ip}, 
\mbox{$\operatorname{cl}_G \hspace{-1pt} W$} is realcompact. Hence,
$G$ is locally realcompact.

(iii) $\Rightarrow$ (iv): Let 
\mbox{$U\hspace{-2.5pt}\in\hspace{-2pt}\mathcal{N}(G)$} be such that
\mbox{$F\hspace{-2pt}:= \hspace{-1pt} \operatorname{cl}_G U$} is 
realcompact. By replacing $U$ with 
\mbox{$U\hspace{-2.25pt}\cap \hspace{-2pt} U^{-1}$} 
if necessary, we may assume that $U$ is {\em symmetric}
(i.e., \mbox{$U \hspace{-2.5pt}= \hspace{-1.75pt} 
U^{-1}\hspace{-0pt}$).}
Put \mbox{$K\hspace{-2.5pt}:=
\hspace{-1pt}\operatorname{cl}_{\widetilde G} U\hspace{-0.5pt}$}
and \mbox{$V\hspace{-2.5pt}:=
\hspace{-0.75pt}\operatorname{int}_{\widetilde G} K\hspace{-2pt}$}.
Since $U$ is symmetric, so are $K$ and $V\hspace{-2.5pt}$.
By Lemma~\ref{real:lemma:zups}, 
\mbox{$F \hspace{-2pt}= \hspace{-2pt}\upsilon F$} is the 
$G_\delta$-closure of $F$ in $\upsilon K\hspace{-1pt}$.
In particular, $F$ is $G_\delta$-closed in $K\hspace{-1.75pt}$.
Let \mbox{$x \hspace{-2pt}\in\hspace{-2pt} \widetilde 
G\backslash G$}. 
We may pick \mbox{$g\hspace{-2pt}\in\hspace{-2pt} (Vx) 
\hspace{-2pt}\cap\hspace{-2pt} G$\hspace{-1pt}}, because
$G$ is dense in $\widetilde G$; one has
\mbox{$x \hspace{-2pt} \in \hspace{-2pt} V\hspace{-1pt}g$}, 
as $V$ is symmetric. 
Since \mbox{$F\hspace{-1pt}g \hspace{-2pt} \subseteq \hspace{-2pt}
G\hspace{-1pt}$} and \mbox{$x\hspace{-2pt} \not\in \hspace{-2pt} 
G\hspace{-1pt}$}, clearly \mbox{$x \hspace{-2pt} \not\in \hspace{-2pt}
F\hspace{-1pt}g$}.
Thus, there is a $G_\delta$-set $A^\prime$ in $\widetilde G$ such  that
\mbox{$x \hspace{-2pt} \in \hspace{-2pt} A^\prime$} and
\mbox{$A^\prime \hspace{-2.5pt} \cap \hspace{-2pt} Fg  \hspace{-2pt} = 
\hspace{-1pt} \emptyset$}, because  $Kg$ is closed in $\widetilde G$ 
and $F\hspace{-1pt}g$ is $G_\delta$-closed in 
$K\hspace{-1pt}g$. Therefore, 
\mbox{$A\hspace{-2pt}: =\hspace{-2pt} 
A^\prime \hspace{-2.5pt} \cap \hspace{-2pt} (V\hspace{-1pt}g)$} is a 
$G_\delta$-set in $\widetilde G$ that contains $x$, and it satisfies
\begin{align}
A \cap G = A^\prime \cap (Vg) \cap G  = A^\prime \cap ((V \cap G) g)
\subseteq A^\prime \cap ((K \cap G) g)
= A^\prime \cap (Fg)  =\emptyset.
\end{align}

\vspace{-3pt}

\noindent
Hence, $G$ is $G_\delta$-closed in $\widetilde G\hspace{-1pt}$, as 
desired.
\end{proof}

The next theorem was inspired by \cite[3.8]{ComfHernTrig}.

\begin{theorem} Let $G$ be a locally\label{real:thm:real}
precompact group. The following statements are equivalent:

\begin{myromanlist}

\item
$G$ is locally realcompact, and 
\mbox{$ip(G)\hspace{-2pt}<\hspace{-2.25pt}\mathfrak{m}$};

\item
$G$ is Dieudonn\'e complete, and 
\mbox{$ip(G)\hspace{-2pt}<\hspace{-2.25pt}\mathfrak{m}$};

\item
$G$ is locally Dieudonn\'e complete, and 
\mbox{$ip(G)\hspace{-2pt}<\hspace{-2.25pt}\mathfrak{m}$};

\item
$G$ is $G_\delta$-closed in $\widetilde G$, and
\mbox{$ip(G)\hspace{-2pt}<\hspace{-2.25pt}\mathfrak{m}$};

\item
$\widetilde G$ is realcompact, and $G$ is $G_\delta$-closed in 
\mbox{$\widetilde G$};

\item
$G$ is realcompact.

\end{myromanlist}
\end{theorem}

\begin{proof}
The equivalences (i) $\Leftrightarrow$ (ii) $\Leftrightarrow$ (iii)
$\Leftrightarrow$ (iv) follow by Theorem~\ref{real:thm:locr}. We note that 
the implication (ii) $\Rightarrow$ (vi) can also be obtained as a 
consequence of Theorem~\ref{real:thm:c-ip}(b).

(iv) $\Rightarrow$ (v): Let $H$ be an $\omega$-precompact open subgroup 
of $G\hspace{-1pt}$, and put \mbox{$M \hspace{-2.5pt} := \hspace{-1pt}
\operatorname{cl}_{\widetilde G} H\hspace{-1.75pt}$.}
(The~\mbox{existence} of such a subgroup $H$ follows from the local
precompactness of $G$; see the proof of Theorem~\ref{real:thm:locr}.) By 
Theorem~\ref{real:thm:ip}(b), $M$ is $\omega$-precompact, and so by 
Theorem~\ref{real:thm:LCLindelof}, $M$ is Lindel\"of. 
Therefore,~by Theorem~\ref{real:thm:bas}(b), $M$ is realcompact. 
Furthermore, by Theorem~\ref{real:thm:ip}(b), 
\mbox{$|\widetilde G\hspace{-0.5pt}/\hspace{-1pt} M| \hspace{-2pt} \leq
\hspace{-2pt} ip(\widetilde G) \hspace{-2pt}= \hspace{-1.5pt} ip(G)$},~%
\mbox{as} 
\mbox{$M \hspace{-2.5pt}\in\hspace{-2pt} \mathcal{H}(\widetilde G)$.} 
Thus, \mbox{$|\widetilde G\hspace{-0.5pt}/\hspace{-1pt} M|
\hspace{-2pt}<\hspace{-2.25pt}\mathfrak{m}$}, and so  by 
Theorem~\ref{real:thm:discrete},
the discrete space 
\mbox{$\widetilde G\hspace{-0.5pt}/\hspace{-1pt} M$}~is~\mbox{realcompact.}
On the other hand, by 
Lemma~\ref{prel:lemma:osubgrp}(a), 
\mbox{$M \hspace{-2.5pt}\in\hspace{-2pt} \mathcal{H}(\widetilde G)$,}
and so
$\widetilde G$ is homeomorphic (as a~topological space)
to \mbox{$M \hspace{-2.5pt} \times \hspace{-1.5pt} (\widetilde G 
\hspace{-0.5pt}/\hspace{-1pt} M)$}. Hence, by 
Remark~\ref{real:rem:rcvsDc}(c), $\widetilde G$ is realcompact, 
being homeomorphic to a product of realcompact spaces.

(v) $\Rightarrow$ (vi): By Lemma~\ref{real:lemma:zups}, 
$\upsilon G$ is the $G_\delta$-closure of $G$ in 
\mbox{$\upsilon \widetilde G \hspace{-2pt} =\hspace{-2pt}
\widetilde G\hspace{-1pt}$.}  Thus, 
\mbox{$\upsilon G \hspace{-2pt} =\hspace{-2pt} G\hspace{-1pt}$.}

(vi) $\Rightarrow$ (i): Since $G$ is realcompact, in particular, it is 
locally realcompact. In order to show that 
\mbox{$ip(G)\hspace{-2pt}<\hspace{-2.25pt}\mathfrak{m}$}, let 
\mbox{$V\hspace{-2pt} \in \hspace{-2pt}\mathcal{N}(G)$}.
Pick an $\omega$-precompact open 
subgroup $H$ of $G$. (The existence~of~such a~subgroup $H$ follows from 
the local  precompactness of $G$, as in the proof of 
Theorem~\ref{real:thm:locr}.) Then 
$H$ can be covered by countably many translates of $V\hspace{-2.5pt}$, and 
so $G$ can be covered by at most 
\mbox{$\omega\hspace{-1pt}\cdot \hspace{-1pt} |G/H|$}-many 
translates~of~$V\hspace{-2.5pt}$. Thus, one has
\mbox{$ip(G)\hspace{-2pt} \leq \hspace{-2pt} 
\omega|G\hspace{-0.5pt}/\hspace{-1pt}H|$}, and it 
suffices to show that 
\mbox{$|G\hspace{-0.5pt}/\hspace{-1pt}H|
\hspace{-2pt}<\hspace{-2.25pt}\mathfrak{m}$}.
Let $X$ be a~set of 
representatives for \mbox{$G\hspace{-0.5pt}/\hspace{-1pt}H\hspace{-1pt}$,} 
that is, \mbox{$|X\hspace{-2.5pt}\cap\hspace{-2pt} (Hg)| 
\hspace{-2pt}=  \hspace{-1.5pt}1$} for every 
\mbox{$g\hspace{-2pt}\in\hspace{-2pt}G\hspace{-1.5pt}$.} Then $X$ is 
discrete and closed in $G$ (because each $Hg$ is open), and consequently,
$X$ is a discrete realcompact space. Hence, by 
Theorem~\ref{real:thm:discrete},
\mbox{$|X| \hspace{-2pt} = \hspace{-2pt}
|G\hspace{-0.5pt}/\hspace{-1pt}H|\hspace{-2pt}<\hspace{-2.25pt}\mathfrak{m}$}, 
as desired.
\end{proof}

\begin{remark}\mbox{ }

\begin{myalphlist}

\item
Suppose that Ulam-measurable cardinals exist, and put
\mbox{$G\hspace{-2pt}:=\hspace{-1.5pt} 
(\mathbb{Z}/2\mathbb{Z})^\mathfrak{m}$}, where $G$ is 
equipped with the product topology. Since $G$ is compact, 
it is realcompact and $\omega$-precompact,~and thus $G$ satisfies all 
conditions  of Theorem~\ref{real:thm:real}, but
\mbox{$|G| \hspace{-2pt}=\hspace{-1.5pt} 2^\mathfrak{m} 
\hspace{-2pt} > \hspace{-2pt} \mathfrak{m}$}. This example shows that 
\begin{myromanlist}[enumii]
\setlength{\itemindent}{25pt}

\item[(i$'$)]
$G$ is locally realcompact, and
\mbox{$|G|<\hspace{-2.25pt}\mathfrak{m}$}

\vspace{6pt}
\end{myromanlist}
cannot be added to the equivalent conditions listed 
in Theorem~\ref{real:thm:real}.

\item
We note in passing the availability of an alternative proof for the 
implication (vi) $\Rightarrow$ (i) in Theorem~\ref{real:thm:real}: If 
\mbox{$ip(G)\hspace{-2pt}\geq \hspace{-2pt}\mathfrak{m}$}, then 
there are \mbox{$U\hspace{-2.5pt} \in\hspace{-2pt} \mathcal{N}(G)$} and a 
(recursively  defined) $\mathfrak{m}$-sequence 
\mbox{$X\hspace{-2pt} =\hspace{-1pt} 
\{x_\eta \mid \eta \hspace{-2pt} < \hspace{-2pt} \mathfrak{m}\}$} in 
$G$ such that \mbox{$x_0 \hspace{-2pt} =  \hspace{-2pt}e$} and
\mbox{$x_\eta\hspace{-1.5pt}\notin\hspace{-2pt}
\bigcup\limits_{\xi<\eta}x_\xi U\hspace{-2pt}$}. 
Then for \mbox{$V\hspace{-2pt}\in\hspace{-2pt}\mathcal{N}(G)$} chosen 
such that \mbox{$V\hspace{-2pt}= \hspace{-2pt}V^{-1}$} and 
\mbox{$V^2 \hspace{-2pt}\subseteq \hspace{-2pt} U\hspace{-1pt}$}, one has
\mbox{$|gV \hspace{-1.5pt}\cap\hspace{-1.5pt} X|
\hspace{-1pt}\leq\hspace{-1pt} 1$} for every 
\mbox{$g \hspace{-2pt} \in \hspace{-2pt}G$}.
Therefore, $X$ is \mbox{discrete}~and closed in $G\hspace{-1pt}$, and of 
non-Ulam-measurable cardinality, contrary to (vi).

\end{myalphlist}
\end{remark}

\section{Hereditary realcompactness}

\label{sect:newHR}

A topological space $X$ is {\em hereditarily realcompact} if every 
subspace of $X$ is realcompact. In this section, we characterize 
hereditary realcompactness in the class of locally precompact groups with 
a ``well-behaved" conjugation structure. 
We rely in this section on the following properties of hereditary  
realcompactness.

\begin{ftheorem}[\pcite{8.18}{GilJer}] \label{newHR:thm:coarser}
If the space $X$ admits a coarser hereditarily 
realcompact topology, then $X$ is hereditarily realcompact.
\end{ftheorem}

Recall that a topological space $X$ has {\em countable pseudocharacter} if 
every singleton in $X$ is a~$G_\delta$-set (cf.~\cite[2.1]{GLCLTG}). 

\begin{ftheorem}[\pcite{8.15}{GilJer}] \label{newHR:thm:HRpsi}
If a space $X$ is realcompact and has  countable pseudocharacter, then 
$X$ is hereditarily realcompact.
\end{ftheorem}

Clearly, if $X$ admits a coarser first-countable topology, then every 
singleton in $X$ is the intersection of countably many open subsets, and 
thus $X$ has countable pseudocharacter. For locally precompact groups, 
the converse is also true.

\begin{theorem} \label{newHR:thm:hommet}
Let $G$ be a locally precompact group. Then $G$ has countable 
pseudocharacter if and only if $G$ admits a coarser homogeneous metrizable 
topology. Moreover, in this case, the metric can be taken to be left invariant.
\end{theorem}

In order to prove Theorem~\ref{newHR:thm:hommet}, we use the following 
classic result (see also the paragraph following the proof of the  
theorem).

\begin{ftheorem}[\pcite{8.14(d)}{HewRos}] \label{newHR:thm:coset}
Let $L$ be a topological group, and $M$ a compact subgroup 
of $L\hspace{-0.25pt}$. Then the coset space 
\mbox{$L\hspace{-0.25pt}/\hspace{-1.25pt}M$}
is metrizable if and only if it is first-countable. 
Moreover, in this case, the metric can be taken to be left invariant.
\end{ftheorem}

\begin{proof}[Proof of Theorem~\ref{newHR:thm:hommet}.]
Since necessity is clear, we focus on sufficiency of the condition. 
Put \mbox{$L\hspace{-2pt}: = \hspace{-1.75pt} 
\widetilde G\hspace{-1pt}$}, and
suppose that $G$ has countable pseudocharacter, that is, $G$ is 
discrete in the $G_\delta\mbox{-}$topology. 
Then there is a $G_\delta$-set $A$ in $L$ such that 
\mbox{$A\hspace{-2pt} \cap \hspace{-2pt} G \hspace{-2pt} = \hspace{-2pt}
\{e\}$}; there exist \mbox{$U_n \hspace{-2.5pt} \in 
\hspace{-2pt} \mathcal{N}(L)$} such that
\mbox{$A \hspace{-2pt} = \hspace{-4pt} \bigcap\limits_{n=1}^\infty
\hspace{-3pt} U_n$}. Since $G$ is locally precompact, its completion $L$ 
is locally compact. Let 
\mbox{$V_0 \hspace{-2.5pt} \in\hspace{-2pt} \mathcal{N}(L)$} be such that
\mbox{$\operatorname{cl}_L \hspace{-1pt} V_0$} is compact. 
For each \mbox{$n\hspace{-2pt}\geq \hspace{-2pt} 1$},
we pick recursively 
\mbox{$V_n \hspace{-2.5pt} \in\hspace{-2pt} \mathcal{N}(L)$} that 
satisfies
\mbox{$V_n V_n \hspace{-2.25pt} \subseteq \hspace{-2pt} 
V_{n-1} \hspace{-2.25pt}\cap\hspace{-2pt} U_n$} and
\mbox{$V_n \hspace{-2.25pt} = \hspace{-1.75pt} V_n^{-1}$}.
Set \mbox{$M \hspace{-2pt} = \hspace{-4pt} \bigcap\limits_{n=1}^\infty
\hspace{-3pt} V_n$}. It is easily seen that $M$ is a closed subgroup;
it is compact, because
\mbox{$M \hspace{-2pt} \subseteq \operatorname{cl}_L \hspace{-1pt} V_0$}.
We turn our attention to the coset space 
\mbox{$L\hspace{-0.25pt}/\hspace{-1.25pt}M\hspace{-2pt}$}. It follows 
from the construction that 
\mbox{$M \hspace{-2.5pt} = \hspace{-3.5pt} 
\bigcap\limits_{n=1}^\infty \hspace{-3pt} (V_n M)$,} and so
\mbox{$L\hspace{-0.25pt}/\hspace{-1.25pt}M$} has countable pseudocharacter.
Since~$L$~is locally compact and the canonical projection
\mbox{$\pi\colon L \rightarrow L\hspace{-0.25pt}/\hspace{-1.25pt}M$}
is open, \mbox{$L\hspace{-0.25pt}/\hspace{-1.25pt}M$} is locally compact too
(cf.~\cite[5.22]{HewRos}).
Therefore, \mbox{$L\hspace{-0.25pt}/\hspace{-1.25pt}M$} is 
first-countable, because every locally compact {\em space} of countable 
pseudocharacter is first-countable (cf.~\cite[3.3.4]{Engel6}).
By Theorem~\ref{newHR:thm:coset}, this implies that
\mbox{$L\hspace{-0.25pt}/\hspace{-1.25pt}M$} is metrizable, and its
metric  can be taken to be left invariant (under the action of $L$). 
Finally, 
it follows from property (ii) that 
\mbox{$M \hspace{-2pt}\subseteq \hspace{-2pt} A$}, and so
\mbox{$M\hspace{-2pt} \cap \hspace{-2pt} G \hspace{-2pt} = \hspace{-2pt}\{e\}$}.
Hence, the restriction $\pi_{|G}$ is injective; its image is metrizable 
and homogeneous, because $G$ acts on it continuously (and transitively) 
from the left. This completes the proof, because the topology
of $\pi(G)$ is the desired coarser homogeneous topology 
generated by a left invariant metric.
\end{proof}

We do not know whether every locally precompact group $G$ with countable 
pseudocharacter admits a coarser metrizable group topology. The answer is 
clearly affirmative if the subgroup $M$ constructed in the proof of 
Theorem~\ref{newHR:thm:hommet} is normal in $L$, because then the quotient 
$L/M$ is itself a~metrizable topological group. This is obviously the case 
when $G$ is abelian. The same conclusion can also be achieved by using a 
Kakutani-Kodaira style argument when $G$ is a (not necessarily abelian) 
$\omega\mbox{-}$precompact group. Indeed, in the latter case, by 
Theorems~\ref{real:thm:ip}(b) and~\ref{real:thm:LCLindelof}, the 
completion $L$ of $G$ is locally compact and $\sigma$-compact.
(For the Kakutani-Kodaira theorem, we refer the reader to the second 
edition of~\cite[8.7]{HewRos}.)

\begin{theorem} \label{newHR:thm:general}
Let $G$ be a locally precompact group. The following statements are 
equivalent:

\begin{myromanlist}

\item
$G$ is hereditarily realcompact;

\item
$G$ has countable pseudocharacter, and 
\mbox{$|G|\hspace{-2pt}  < \hspace{-2pt}  \mathfrak{m}$;}

\item
$G$ has countable pseudocharacter, and 
\mbox{$ip(G)\hspace{-2pt} < \hspace{-2pt} \mathfrak{m}$;}

\item
$G$ admits a coarser homogeneous metrizable topology, and 
\mbox{$|G|\hspace{-2pt} < \hspace{-2pt} \mathfrak{m}$;}

\item
$G$ admits a coarser homogeneous metrizable topology, and 
\mbox{$ip(G)\hspace{-2pt} < \hspace{-2pt} \mathfrak{m}$.}

\end{myromanlist}
\end{theorem}

\begin{proof}
The equivalences (ii)  $\Leftrightarrow$ (iv) and 
(iii) $\Leftrightarrow$ (v)  follow by Theorem~\ref{newHR:thm:hommet}, 
while (ii) $\Rightarrow$ (iii) is clear, because
\mbox{$ip(G) \hspace{-2pt}\leq\hspace{-2pt} |G|$.}

(i) $\Rightarrow$ (ii):
If $G$ is discrete, then clearly it has countable pseudocharacter, and so 
we may assume without loss of generality that $G$ is not discrete. Let
\mbox{$g \hspace{-2pt} \in\hspace{-2pt} G \hspace{-1pt}$,} and put 
\mbox{$X\hspace{-2.75pt}: =\hspace{-1.5pt} 
G\backslash\{g\}\hspace{-1pt}$.}
Since~$G$~is not discrete, $X$ is dense in $G\hspace{-1pt}$. By 
Corollary~\ref{real:cor:LPOz}, $G$ is  an 
$Oz$-space, and thus by Theorem~\ref{real:thm:Oz}, $X$ is 
$z$-embedded in $G\hspace{-1pt}$. Therefore, by 
Theorem~\ref{real:thm:BlaHag}(b), 
$\upsilon X$ is the $G_\delta$-closure of $X$ in $\upsilon 
G\hspace{-1pt}$. 
By (i), both $X$ and $G$ are realcompact, and so $X$ is $G_\delta$-closed 
in $G\hspace{-1pt}$. Hence, 
\mbox{$G\backslash X \hspace{-2.25pt} =\hspace{-2.25pt}\{g\}$}
is $G_\delta$-open. Since every
$G_\delta$-open singleton is a $G_\delta$-set, the group $G$ has countable 
pseudocharacter.

The $G_\delta$-topology on $G$ is finer than the topology of $G$, and 
so by Theorem~\ref{newHR:thm:coarser}, the $G_\delta\mbox{-}$topology on $G$ 
is hereditarily realcompact. On the other hand, since $G$ has countable 
pseudocharacter, the $G_\delta$-topology is discrete on $G\hspace{-1pt}$. 
Therefore, by Theorem~\ref{real:thm:discrete}, 
\mbox{$|G| \hspace{-2pt}<\hspace{-2pt} \mathfrak{m}$}, as desired.

(iii) $\Rightarrow$ (i):  The group $G$ equipped with the 
$G_\delta$-topology is discrete, because it has countable 
pseudocharacter. Thus, $G$ is $G_\delta$-closed in 
$\widetilde G\hspace{-1pt}$, since (in every topological group)
every discrete {\em subgroup} is 
closed (cf.~\cite[1.51]{GLCLTG}). 
Therefore, by Theorem~\ref{real:thm:real},
$G$ is realcompact. Hence, by Theorem~\ref{newHR:thm:HRpsi}, $G$ is 
hereditarily realcompact.
\end{proof}

If $D$ is a discrete space such that 
\mbox{$\omega \hspace{-2pt}<  \hspace{-2pt} |D|
 \hspace{-2pt} < \hspace{-2pt} \mathfrak{m}$}, then 
the Alexandroff
one-point compactification of $D$ is compact, hereditarily 
realcompact (by Theorem~\ref{real:thm:discrete}), but not metrizable.
It follows from Corollary~\ref{newHR:cor:LC}(b) below that no such example 
exists among topological groups. 
Since every locally compact {\em space} of countable pseudocharacter is 
first-countable (cf.~\cite[3.3.4]{Engel6}), every locally compact {\em 
group} of countable pseudocharacter is metrizable
(cf.~\cite[1.23]{GLCLTG}). Thus, Theorem~\ref{newHR:thm:general} has the 
following consequence:

\begin{corollary} \label{newHR:cor:LC}
Let $L$ be a  locally compact  group.  Then: 

\begin{myalphlist}

\item
$L$ is hereditarily realcompact if and only if it is metrizable and  
\mbox{$|L|\hspace{-2pt}<\hspace{-2pt}\mathfrak{m}$}; 

\item
if $L$ is Lindel\"of, then $L$ is hereditarily realcompact if 
and only if it is metrizable.
\qed

\end{myalphlist}
\end{corollary}

Theorem~\ref{newHR:thm:general} guarantees only the existence of a coarser 
homogeneous metrizable topology, but falls short of providing a coarser 
metrizable group topology.  So far as we are aware, such a group topology 
is available only under some additional assumptions on the algebraic and 
topological structure of the group.

A topological group $G$ is said to be
{\em $\omega$-balanced} if for every \mbox{$U \hspace{-2.5pt}\in  
\hspace{-2pt}\mathcal{N}(G)$,} there is 
\mbox{$\mathcal{V}_U  \hspace{-2pt}\subseteq \hspace{-2pt} 
\mathcal{N}(G)$} such 
that for every \mbox{$x  \hspace{-2pt}\in \hspace{-2pt} G$,} there is 
\mbox{$V  \hspace{-2.5pt}\in  \hspace{-2pt}\mathcal{V}_U$}
that satisfies \mbox{$x^{-1}Vx  \hspace{-2pt}\subseteq 
\hspace{-2pt} U\hspace{-1.5pt}$,} and
\mbox{$|\mathcal V_U|  \hspace{-2pt}\leq 
\hspace{-2pt}\omega\hspace{-1pt}$}
(cf.~\cite[2.7]{GLCLTG}). The class of $\omega$-balanced groups was 
introduced by Kac (under the name of {\em groups with a~quasi-invariant 
basis}), who also proved that a group is $\omega$-balanced 
if and only if it embeds as a topological group into a product of 
metrizable groups (cf.~\cite{Kac} and~\cite[2.18]{GLCLTG}).
(For the sake of correct historical presentation, we note that
questions related to embedding of topological groups into the product of 
groups of a certain class were first studied by Graev~\cite{Graev};
Kac's results were generalized later by Arhangel'ski\v{\i}~\cite{Arh}
and Guran~\cite{Guran}.) Thanks to the following theorem due to Kac, 
$\omega$-balanced groups lend themselves to a more elegant 
characterization of hereditary realcompactness.

\begin{ftheorem}[\cite{Kac}, \pcite{2.19}{GLCLTG}] \label{newHR:thm:Kac}
Let $G$ be an $\omega$-balanced topological group. Then $G$ has countable 
pseudocharacter if and only if $G$ admits a coarser metrizable group 
topology.
\end{ftheorem}

\begin{discussion} 
The class of $\omega$-balanced groups contains all abelian groups,
metrizable groups, $\omega\mbox{-}$precompact groups (cf.~\cite[2.27]{GLCLTG}), 
and also the so-called {\em balanced groups} (i.e., groups whose left and 
right uniform structures coincide; cf.~\cite[1.25]{GLCLTG}). 
By Theorem~\ref{newHR:thm:Kac}, if $G$ is an $\omega$-balanced 
locally precompact group, then the conditions

\begin{myromanlist}
\setlength{\itemindent}{25pt}

\item[(iv$'$)]
$G$ admits a coarser metrizable group topology, and  
\mbox{$|G|\hspace{-2pt}<\hspace{-2pt} \mathfrak{m}$}, and

\item[(v$'$)]
$G$ admits a coarser metrizable group topology, and
\mbox{$ip(G)\hspace{-2pt}<\hspace{-2pt} \mathfrak{m}$},

\vspace{6pt}

\end{myromanlist}

\noindent
may be added to the equivalent conditions  listed in 
Theorem~\ref{newHR:thm:general}.
\end{discussion}

\pagebreak[3]

Since every metrizable precompact group has cardinality at most 
$\mathfrak{c}$, Theorem~\ref{newHR:thm:general} can be stated in a simple 
form for precompact groups, and it  implies 
portions of \cite[4.6]{ComfHernTrig} and \cite[3.3]{HernMaca}.

\begin{corollary}
For every precompact group $G$, the following statements are equivalent:

\begin{myromanlist}

\item
$G$ is hereditarily realcompact;

\item
$G$ has countable pseudocharacter;

\item
$G$ admits a coarser metrizable group topology, and 
\mbox{$|G|\hspace{-2.25pt}\leq\hspace{-2pt} \mathfrak{c}$}.
\qed

\end{myromanlist}
\end{corollary}

\section{Connectedness properties}

\label{sect:conn}

With Theorems~\ref{conn:thm:LC}  and~\ref{conn:thm:0dimqLC} below in mind 
as motivation, we investigate 
in this section the relationship between connectedness properties of
locally precompact (or locally pseudocompact) groups and their 
completions.

\begin{notation}
With each topological group $G$ are associated functorial subgroups 
related to connectedness properties of $G\hspace{-1.5pt}$, defined as 
follows (cf.~\cite[1.1.1]{Dikconcomp}):
\begin{myalphlist}

\item
$G_0$ denotes the connected component of the identity;

\item
$q(G)$ denotes the {\em quasi-component} of the identity, that is, the 
intersection of all clopen sets containing the identity;

\item
\mbox{$o(G) \hspace{-2pt}:= \hspace{-2pt}\bigcap \mathcal{H}(G)$}, the 
intersection of all open subgroups of $G\hspace{-1.5pt}$.

\vskip 1pt

\end{myalphlist}
It is well known and easily seen
that all three of these subgroups are closed and normal
(cf.~\cite[7.1]{HewRos}, \cite[2.2]{DikCOTA}, and~\cite[1.32(b)]{GLCLTG}).
Clearly, \mbox{$G_0\hspace{-2pt}\subseteq \hspace{-2pt}
q(G) \hspace{-2pt} \subseteq \hspace{-2pt}o(G)$}, and for locally compact 
groups, all three are equal:
\end{notation}

\begin{ftheorem}[\pcite{7.8}{HewRos}] \label{conn:thm:LC}
Let $L$ be a locally compact group. Then 
\mbox{$L_0\hspace{-2pt} = \hspace{-2pt}
q(L) \hspace{-2pt} = \hspace{-2pt}o(L)$}.
\end{ftheorem}

Following many authors, we say that a space is {\itshape zero-dimensional} 
if it has a base consisting of {\itshape clopen} (open-and-closed)  sets. 
It is clear that a zero-dimensional (Hausdorff) group $G$ satisfies 
\mbox{$q(G)\hspace{-2pt} = \hspace{-2pt} \{e\}$}.

\begin{ftheorem}[{\cite[3.5, 7.13]{HewRos}}] \label{conn:thm:0dimqLC}
Let $L$ be a locally compact group, and $N$ a closed normal subgroup. Then 
the following statements are equivalent:

\begin{myromanlist}

\item
$L/N$ is zero-dimensional;

\item
\mbox{$(L/N)_0 \hspace{-2pt} = \hspace{-2pt} \{N\}$};

\item
\mbox{$L_0 \hspace{-2pt} \subseteq \hspace{-2pt} N \hspace{-2pt}$}.

\end{myromanlist}
\end{ftheorem}

One may wonder whether the conclusions of  Theorems~\ref{conn:thm:LC} 
and~\ref{conn:thm:0dimqLC} hold for locally precompact groups. 
Examples~\ref{ex:conn:long}(a)-(e) provide a negative 
answer to this question.
Moreover, as Example~\ref{ex:conn:long}(d) and Theorem~\ref{LPAq:thm:main} 
indicate,  the relation \mbox{$G_0\hspace{-2pt} = \hspace{-2pt} q(G)$} 
fails for some pseudocompact abelian groups. When $G$ and 
$H$ are topological groups, we use the symbol 
\mbox{$G \hspace{-2pt} \cong \hspace{-2pt} H$} to indicate that $G$ and 
$H$ are topologically isomorphic, that is,  there is  a bijection from $G$ 
onto $H$ that is simultaneously an algebraic isomorphism and a 
topological homeomorphism.

\begin{examples} \mbox{ } \label{ex:conn:long}

\begin{myalphlist}

\item
Comfort and van Mill showed that there exists a pseudocompact abelian 
group $G$ such that \mbox{$G_0 \hspace{-2pt}=\hspace{-1.5pt}
q(G)\hspace{-2pt}=\hspace{-2pt}\{0\}$}, but $G$ is 
not zero-dimensional (cf.~\cite[7.7]{ComfvMill2}). 
Thus, Theorem~\ref{conn:thm:0dimqLC} fails~not only for locally precompact 
\pagebreak[3]
groups (or locally pseudocompact ones), but even for pseudocompact groups.
Nevertheless, it is possible to characterize zero-dimensional quotients of 
locally pseudocompact groups (Theorem~\ref{conn:thm:lcps0dimq}).

\item
Ursul showed that there is a subgroup
$G$ of the group $\mathbb{R}^2$ in its usual topology such that
\mbox{$G_0\hspace{-2pt} = \hspace{-2pt} \{0\}$} and
\mbox{$q(G)\hspace{-2pt} \cong \hspace{-2pt} \mathbb{Z}$} 
(cf.~\cite{UrsulExpl}). Thus, the equality
\mbox{$G_0\hspace{-2pt} = \hspace{-2pt} q(G)$} in 
Theorem~\ref{conn:thm:LC} fails for the (locally precompact) group $G$.

\item
Put \mbox{$G\hspace{-2pt} = \hspace{-1.5pt} \mathbb{Q}/\mathbb{Z}$}. It is 
obvious (and also follows from Theorem~\ref{conn:thm:3stat} below) that $G$ 
has no proper open subgroups, and so 
\mbox{$o(G) \hspace{-2pt} = \hspace{-1.5pt} G\hspace{-1pt}$}. On the other 
hand, like every Tychonoff space of cardinality less than continuum, 
$G$ is zero-dimensional 
(the argument given in \cite[6.2.8]{Engel6} suffices to show this); in 
particular, 
\mbox{$q(G) \hspace{-2pt} = \hspace{-2pt} \{0\}\hspace{-1pt}$}. Therefore, 
the equality \mbox{$q(G)\hspace{-2pt} = \hspace{-2pt} o(G)$}
in Theorem~\ref{conn:thm:LC} fails for the (precompact) group $G$.

\item
By Theorem~\ref{LPAq:thm:main}(b) below, there is a pseudocompact 
abelian 
group $G$ such that 
\mbox{$q(G)\hspace{-2pt}\cong\hspace{-2pt} \mathbb{Z}/2\mathbb{Z}$}. 
Since \mbox{$G_0 \hspace{-2pt} \subseteq \hspace{-2pt} q(G)$} and
$q(G)$ is discrete, it follows that 
\mbox{$G_0\hspace{-1.5pt} =\hspace{-2pt} \{0\}\hspace{-1pt}$}.
Thus, the equality \mbox{$G_0 \hspace{-2pt} = \hspace{-2pt} q(G)$} in 
Theorem~\ref{conn:thm:LC} fails even for some pseudocompact abelian 
groups.

\item
The iterated quasi-components $q(G), q(q(G)),\ldots,q_\alpha(G)$ 
of a topological group $G$ define a descending chain of normal 
subgroups of $G\hspace{-1pt}$ indexed by ordinals. Dikranjan showed that 
for every ordinal $\alpha$, there is a pseudocompact abelian group 
$H$ such that 
\mbox{$H_0 \hspace{-1.6pt} = \hspace{-1.5pt} q_\alpha(H)$}, but
\mbox{$H_0 \hspace{-1.6pt} \subsetneq \hspace{-1.5pt} q_\beta(H)$} for 
every \mbox{$\beta \hspace{-2pt} < \hspace{-2pt} \alpha$}
(cf.~\cite[Theorem~11]{Dikdimpsc} and~\cite[1.4.10]{Dikconcomp}).
Dikranjan's construction~is~the most striking illustration known to the 
authors of how big the gap between $G_0$ and $q(G)$ can be.

\end{myalphlist}
\end{examples}

We show now that the equality \mbox{$q(G)\hspace{-2pt} = \hspace{-2pt} 
o(G)$} in Theorem~\ref{conn:thm:LC} does hold for locally pseudocompact 
groups. The following theorem generalizes \cite[1.4]{Dikconpsc}, which 
treats the same property in the case of pseudocompact groups.

\begin{theorem} \label{conn:thm:qqoo}
Let $G$ be a locally pseudocompact group. Then:

\begin{myalphlist}

\item
\mbox{$q(G)\hspace{-2.1pt}=\hspace{-1.9pt} 
q(\widetilde G) \hspace{-2.1pt} \cap\hspace{-2pt} G$};

\item
\mbox{$q(G)\hspace{-2.1pt}=\hspace{-1.9pt} 
(\widetilde G)_0 \hspace{-2.1pt} \cap\hspace{-2pt} G$};

\item
\mbox{$q(G)\hspace{-2.1pt}=\hspace{-1.9pt} o(G)$}.

\end{myalphlist}
\end{theorem}

\begin{proof}
(a) For every Tychonoff space  $X\hspace{-1.5pt}$,
the quasi-component of \mbox{$x \hspace{-2pt}\in\hspace{-2pt} X$}
is equal to the trace~on~$X$ of the quasi-component of $x$ in 
$\beta X \hspace{-2pt}$ (cf.~\cite[2.1]{DikCLP}).
By the implication (i) $\Rightarrow$ (v) of 
Theorem~\ref{prel:thm:lcps},~one has
\mbox{$\beta G \hspace{-2pt} = 
\hspace{-2pt} \beta \widetilde G\hspace{-1pt}$}, and the statement 
follows.

(b) As $\widetilde G$ is locally compact, by Theorem~\ref{conn:thm:LC},
\mbox{$q(\widetilde G) \hspace{-2.1pt} = \hspace{-1.9pt} 
(\widetilde G)_0$}. Thus, the statement follows by~(a).

(c) Since $\widetilde G$ is locally compact, one has
\mbox{$q(\widetilde G) \hspace{-2.1pt} = \hspace{-1.9pt} o(\widetilde G)$}
by  Theorem~\ref{conn:thm:LC}.
By  (a) and Lemma~\ref{prel:lemma:osubgrp}(d),
\begin{align}
q(G) = q(\widetilde G) \cap G = o(\widetilde G) \cap G = o(G),
\end{align}
as desired.
\end{proof}

\begin{theorem} \label{conn:thm:3stat}
Let $G$ be a locally precompact group, and consider the following statements:

\begin{myromanlist}

\item
$G$ is connected;

\item
$G$ has no proper open subgroups;

\item
$\widetilde G$ is connected.

\vskip 1pt
\end{myromanlist}
Then {\rm (i)} $\Rightarrow$ {\rm (ii)} $\Leftrightarrow$ {\rm (iii)}.
Furthermore, if $G$ is locally pseudocompact, then 
all three conditions are equivalent.
\end{theorem}

In order to prove Theorem~\ref{conn:thm:3stat}, we rely on a well-known 
relationship between the connectedness of a space and its Stone-\v{C}ech 
compactification:

\pagebreak[3]

\begin{ftheorem}[\pcite{6L.1}{GilJer}] \label{conn:thm:betaX}
A Tychonoff space $X$ is  connected if and only if $\beta X$ is 
connected.
\end{ftheorem}

\begin{proof}[Proof of Theorem~\ref{conn:thm:3stat}.]
The implication (i) $\Rightarrow$ (ii) is obvious, because
every open subgroup is closed.

(ii) $\Rightarrow$ (iii): If $G$ has no proper open subgroups, then
\mbox{$o(\hspace{-0.6pt} G)\hspace{-2pt} = \hspace{-1.5pt} G\hspace{-1pt}$}. 
Consequently, 
\mbox{$G\hspace{-2pt}\subseteq\hspace{-2pt}o(\hspace{-0.6pt}\widetilde G)$}
by Lemma~\ref{prel:lemma:osubgrp}(d), and so 
\mbox{$G\hspace{-2pt}\subseteq\hspace{-2pt}(\widetilde G)_0$}
by Theorem~\ref{conn:thm:LC}.
Since $G$ is dense in $\widetilde G$ and $(\widetilde G)_0$ is a closed 
subgroup, this implies that 
\mbox{$(\widetilde G)_0 \hspace{-2pt} = 
\hspace{-1.5pt} \widetilde G \hspace{-1pt}$}.

(iii) $\Rightarrow$ (ii): If $\widetilde G$ is connected, then by
Theorem~\ref{conn:thm:LC}, 
\mbox{$o(\hspace{-0.6pt}\widetilde G) \hspace{-2.3pt} = 
\hspace{-1.5pt} \widetilde G \hspace{-1pt}$}, and so by 
Lemma~\ref{prel:lemma:osubgrp}(d), 
\mbox{$o(\hspace{-0.6pt}G)\hspace{-2.15pt} = \hspace{-1.5pt} 
G\hspace{-1pt}$}.

(i) $\Leftrightarrow$ (iii): If $G$ is locally pseudocompact,
then by the implication (i) $\Rightarrow$ (v) of
Theorem~\ref{prel:thm:lcps}, one has
\mbox{$\beta G \hspace{-2pt} =
\hspace{-2pt} \beta \widetilde G\hspace{-1pt}$}, and thus
the statement follows by Theorem~\ref{conn:thm:betaX}.
\end{proof}

Connectedness is not the only property that holds for a locally 
pseudocompact group if and only if it holds for its completion. The
same is true for the other extreme, namely, zero-dimensionality.

\begin{theorem} \label{conn:thm:0dim-compl}
Let $G$ be a locally pseudocompact group. Then 
$G$ is zero-dimensional if and only if 
$\widetilde G$ is zero-dimensional.
\end{theorem}

\begin{proof}
Suppose that $G$ is zero-dimensional. Then its topology has a clopen base 
at $e$, and thus $\beta G$ has a clopen base at $e$ 
(cf.~\cite[6L.2]{GilJer}). 
Since $G$ is locally pseudocompact, by the implication 
(i)~$\Rightarrow$~(v) of Theorem~\ref{prel:thm:lcps}, one has
\mbox{$\beta G \hspace{-2pt} = 
\hspace{-2pt} \beta \widetilde G\hspace{-1pt}$}. Therefore, 
$\widetilde G$ admits a clopen base at $e$. This shows that $\widetilde G$ 
is zero-dimensional. Since $G$ is a subspace of $\widetilde G$, the converse 
is obvious.
\end{proof}

In connection with the proof of Theorem~\ref{conn:thm:0dim-compl},
it is well to recognize that although the zero-dimensional property is 
inherited by all subspaces, there are zero-dimensional spaces $X$ for 
which $\beta X$ is not zero-dimensional (cf.~\cite[6.2.20]{Engel6}
and~\cite[16P.3]{GilJer}). In particular, a zero-dimensional space in 
our terminology need not have {\it Lebesgue covering dimension} zero.

We already noted in Example~\ref{ex:conn:long}(a) that the conclusion of
Theorem~\ref{conn:thm:0dimqLC} fails for certain (locally) pseudocompact 
groups. Nevertheless, it is possible to obtain a meaningful 
characterization of zero-dimensional quotients of such groups.

\begin{theorem} \label{conn:thm:lcps0dimq}
Let $G$ be a locally pseudocompact group, and $M$ a closed normal 
subgroup. Then $G/M$ is zero-dimensional if and only if \
\mbox{$(\widetilde G)_0 \hspace{-1.5pt}\subseteq \hspace{-1.5pt}
\operatorname{cl}_{\widetilde G} \hspace{-1pt} M\hspace{-1pt}$.}
\end{theorem}

In the proof of Theorem~\ref{conn:thm:lcps0dimq}, we rely on the following 
lemma, which is a variant on a theorem of Sulley  that was 
extended to the non-abelian case by Grant (cf.~\cite{SulleyOMT} 
and~\cite[1.3]{GrantOMT}).

\begin{flemma}[\pcite{1.19}{GLCLTG}] \label{conn:lemma:dense-emb}
Let $G$ be a topological group,\hspace{-0.75pt}
$D$ a dense subgroup,\hspace{-0.0pt} and $M$\hspace{-0.5pt}
a closed normal subgroup of $D\hspace{-1.5pt}$. Then 
\mbox{$N\hspace{-2.6pt}= \hspace{-1pt}\operatorname{cl}_G M$} is a 
normal subgroup of $G$, and the canonical homomorphism
$\bar\pi_{|D}\colon D/M \rightarrow DN/N$ is a topological isomorphism.
\end{flemma}

\begin{proof}[Proof of Theorem~\ref{conn:thm:lcps0dimq}.]
Put \mbox{$N\hspace{-2.6pt}= \hspace{-1pt}
\operatorname{cl}_{\widetilde G} M\hspace{-2pt}$}.
By Lemma~\ref{conn:lemma:dense-emb}, with $\widetilde{G}$ and $G$ 
replacing $G$ and $D$, respectively, there is a topological isomorphism
$\varphi$ from $G/M$ onto a dense subgroup of the locally compact group
\mbox{$\widetilde G / N\hspace{-1.5pt}$}.  Thus,
\mbox{$\widetilde G / N$} is the completion of 
\mbox{$\varphi(G/M)$}.
By the implication (i)~$\Rightarrow$~(vii) of 
Theorem~\ref{prel:thm:lcps},  $G$ is $G_\delta$-dense in 
$\widetilde G \hspace{-1pt}$. Consequently, the image 
\mbox{$\varphi(G/M)$} is  $G_\delta$-dense 
in~\mbox{$\widetilde G / N\hspace{-1.5pt}$}, and so
by Theorem~\ref{prel:thm:lcps},  \mbox{$\varphi(G/M)$} is also 
locally pseudocompact. Therefore, by 
Theorem~\ref{conn:thm:0dim-compl},
\mbox{$\varphi(G/M)$} is zero-dimensional if and only if 
\mbox{$\widetilde G / N$} is zero-dimensional.
By Theorem~\ref{conn:thm:0dimqLC}, the latter holds if and only if 
\mbox{$(\widetilde G)_0 \hspace{-1.5pt}\subseteq \hspace{-1.5pt}
N\hspace{-2pt}$.}
\end{proof}

\begin{corollary} \label{conn:cor:G/qG0}
Let $G$ be a locally pseudocompact group. Then:

\begin{myalphlist}

\item
$G/q(G)$ is zero-dimensional if and only if $q(G)$ is dense in 
$(\widetilde G)_0$;

\item
$G/G_0$ is zero-dimensional if and only if $G_0$ is dense in 
$(\widetilde G)_0$, in which case 
\mbox{$G_0 \hspace{-2pt} = \hspace{-1.5pt} q(G)$}. \qed

\end{myalphlist}
\end{corollary}


\begin{discussion}
Theorems~\ref{conn:thm:0dim-compl} and~\ref{conn:thm:lcps0dimq} were 
inspired by the work of Dikranjan (cf.~\cite{DikPS0dim}), and
Corollary~\ref{conn:cor:G/qG0} generalizes \cite[1.7]{DikPS0dim}.
Dikranjan showed that if every closed subgroup of $G$ is pseudocompact, then 
$G/G_0$ is zero-dimensional and $G_0$ is dense in $(\widetilde G)_0$.
(cf.~\cite[1.2]{DikPS0dim}). It is natural to ask whether a~similar
statement is true if one replaces ``pseudocompact" with ``locally 
pseudocompact."
\end{discussion}

\begin{problem} \label{conn:prob}
Let $G$ be a topological group such that every closed subgroup of $G$ 
is locally pseudocompact. Is $G/G_0$ zero-dimensional? Equivalently, is 
$G_0$ dense in $(\widetilde G)_0$?
\end{problem}

After the present manuscript was submitted, Dikranjan and Luk\'acs 
provided a positive answer to Problem~\ref{conn:prob} 
(cf.~\cite[Theorem~A]{DikGL3}).

\bigskip

We turn now to connectedness in the local context. 
Recall that  a space $X$ is {\em locally connected} if 
each connected component of every open subspace of $X$ is open.
The proof of the following easy lemma is omitted.

\begin{lemma} \label{conn:lemma:denseLocCon}
Let $G$ be a topological group, and $D$ a dense subgroup.
If $D$ is locally connected, then so is $G$.
\qed
\end{lemma}

\begin{theorem} \label{conn:thm:lcpsLocCon}
Let $G$ be a locally pseudocompact group. Then $G$ is locally connected if 
and only if\hspace{1.5pt}  $\widetilde G$ is locally connected.
\end{theorem}

\begin{proof}
Lemma~\ref{conn:lemma:denseLocCon} proves the implication
$\Rightarrow$.  For $\Leftarrow$, let \mbox{$V \hspace{-2.5pt} \in 
\hspace{-2pt} \mathcal{N}(G)$}. There is \mbox{$W \hspace{-2.5pt} \in 
\hspace{-2pt} \mathcal{N}(G)$} such that
\mbox{$\operatorname{cl}_G W  \hspace{-2pt}\subseteq\hspace{-2pt} V$} 
and $W$ is precompact. Then there is 
\mbox{$W^\prime \hspace{-2.5pt}\in\hspace{-2pt}\mathcal{N}(\widetilde G)$} 
such that 
\mbox{$W \hspace{-3pt}=  \hspace{-1pt}W^\prime
\hspace{-2.1pt}\cap\hspace{-2pt} G\hspace{-1pt}$}, 
and thus
\begin{align}
W^\prime \subseteq \operatorname{cl}_{\widetilde G} W^\prime
= \operatorname{cl}_{\widetilde G} (W^\prime\cap G) = 
\operatorname{cl}_{\widetilde G} W.
\end{align}
Let $C$ denote the connected component of the identity in $W^\prime$. 
Since $\widetilde G$ is locally connected, one has
\mbox{$C \hspace{-2pt}\in\hspace{-2pt} \mathcal{N}(\widetilde G)$}.
Consequently, 
\mbox{$U \hspace{-2.75pt} = \hspace{-1.25pt} C\hspace{-2pt}\cap 
\hspace{-2pt} G \hspace{-2pt}\in\hspace{-2pt} \mathcal{N}(G)$,} 
and  
\mbox{$U \hspace{-2.5pt} \subseteq \hspace{-2pt} 
W^\prime\hspace{-3pt} \cap \hspace{-2pt} G 
\hspace{-1.75pt}=  \hspace{-1.5pt}W$}; 
in particular, $U$~is precompact.
Since $G$ is locally pseudocompact, using
Theorem~\ref{prel:thm:lcps}, we obtain that
\begin{align}
\beta (\operatorname{cl}_G U) = \operatorname{cl}_{\widetilde G} U = 
\operatorname{cl}_{\widetilde G} (C \cap G) = 
\operatorname{cl}_{\widetilde G} C,
\end{align}
which is connected, being a closure of the connected set $C$.
Therefore, by Theorem~\ref{conn:thm:betaX}, 
$\operatorname{cl}_G U$ is connected. Finally, observe that
\mbox{$e \hspace{-2pt}\in\hspace{-2pt} U \hspace{-2.25pt}\subseteq 
\hspace{-1.5pt} \operatorname{cl}_G U \hspace{-2pt}\subseteq \hspace{-1.5pt}
\operatorname{cl}_G W \hspace{-2pt}\subseteq \hspace{-2pt}
V\hspace{-2pt}$,} as desired.
\end{proof}

\begin{example}
Let \mbox{$G = \mathbb{Q}/\mathbb{Z}$.}  Clearly, \mbox{$\widetilde G = 
\mathbb{R}/\mathbb{Z}$} is compact, connected, and locally connected;~in 
particular, $G$ is precompact. It is obvious (and also follows from 
Theorem~\ref{conn:thm:3stat}) that the group $G$ 
has no proper open subgroups. However, $G$ is 
zero-dimensional (cf.~\cite[6.2.8]{Engel6}). 
This shows that the  assumption of local pseudocompactness in 
Theorems~\ref{conn:thm:3stat} (the implication (ii) $\Rightarrow$ (i))
and~\ref{conn:thm:lcpsLocCon} cannot be 
omitted even for precompact groups.
\end{example}

\section{Which locally precompact abelian groups occur as a 
quasi-component?}

\label{sect:LPAq}

In order to answer the question in the title of this section, some further 
terminology is required. Recall that a topological group is {\em compactly 
generated} if it is generated algebraically by some compact subset. Since 
every connected group is generated by every neighborhood of its identity 
(cf.~\cite[1.30]{GLCLTG}), every connected locally compact group is 
compactly generated.

We use additive notation for abelian topological groups. We put
\mbox{$\mathbb{T}\hspace{-2pt}= \hspace{-1pt}\mathbb{R}/\mathbb{Z}$,}
which is the circle group written additively.
Recall that if $p$ is a prime number, then the group $\mathbb{Z}_p$ of 
{\em $p$-adic integers} is the (projective) limit of the quotients
$\mathbb{Z}/p^n \mathbb{Z}$. The group $\mathbb{Z}_p$ is compact,
zero-dimensional, and $\{p^n \mathbb{Z}_p\}_{n \in \mathbb{N}}$ is
a~base of open subgroups for the topology at zero 
(cf.~\cite[\S10]{HewRos} and~\cite[\S3.5]{DikProSto}).

\begin{definition}
A topological group $G$ is {\em precompactly generated} if there is a 
precompact set \mbox{$X\hspace{-2.5pt}\subseteq\hspace{-2pt} G$} such that 
\mbox{$G\hspace{-1.75pt} =\hspace{-2pt}\langle X \rangle$}.
\end{definition}

\begin{lemma} \label{LPAq:lemma:charpcg}
Let $G$ be a locally precompact group. Then $G$ is precompactly generated 
if and only if $\widetilde G$ is compactly generated.
\end{lemma}

\begin{proof}
Let \mbox{$X\hspace{-2.5pt}\subseteq\hspace{-2pt} G$} be a precompact set 
such that \mbox{$G\hspace{-1.75pt} =\hspace{-2pt}\langle X \rangle$}. 
Then, by Theorem~\ref{prel:thm:charpr}, 
\mbox{$Y\hspace{-2.25pt}: =\hspace{-1pt}\operatorname{cl}_{\widetilde G} X$} 
is compact. Since $G$ is locally precompact, $\widetilde G$ is locally 
compact, and so there is \mbox{$V \hspace{-2.75pt} \in \hspace{-2pt} 
\mathcal{N}(\widetilde G)$} such that 
\mbox{$\operatorname{cl}_{\widetilde G} V$} is compact.
It is easy to see that the compact set 
\mbox{$K\hspace{-2.25pt} := \hspace{-1.2pt} 
(\operatorname{cl}_{\widetilde G} V)Y$} generates $\widetilde 
G\hspace{-1pt}$. 

Conversely, suppose that $\widetilde G$ is compactly generated, that is,
\mbox{$\widetilde G \hspace{-1.75pt}= \hspace{-2.1pt} \langle K \rangle$}, 
where $K$ is compact. Since $G$ is locally precompact, its completion
$\widetilde G$ is 
locally compact, and so there is 
\mbox{$V \hspace{-2.75pt} \in \hspace{-2pt} \mathcal{N}(\widetilde G)$}~such 
that \mbox{$\operatorname{cl}_{\widetilde G} V$} is compact. 
Put \mbox{$U \hspace{-2.5pt} := \hspace{-1.75pt} (VK) 
\hspace{-2.1pt} \cap \hspace{-1.5pt} G\hspace{-1pt}$}. As
$VK$ is open in $\widetilde G\hspace{-1pt}$, the set $U$ is open in 
$G\hspace{-1pt}$. The set~$U$~is precompact, because it is contained in 
the compact set  
\mbox{$\operatorname{cl}_{\widetilde G} (VK) \hspace{-2pt}= \hspace{-1pt}
(\operatorname{cl}_{\widetilde G} V)K$}.  It is easily seen 
that $U$ generates~$G$. 
\end{proof}

For a topological space $X\hspace{-1.5pt}$, we denote by $w(X)$ the {\em 
weight} of $X\hspace{-1.5pt}$, that is, the smallest possible cardinality 
of a base for the topology of~$X\hspace{-1pt}$.

\begin{theorem}  \label{LPAq:thm:pcg}
Let $G$ be a locally precompact abelian group. The 
following statements are equivalent:

\begin{myromanlist}

\item
$G$ is precompactly generated;

\item
$G$ is topologically isomorphic to a subgroup of a connected locally 
compact abelian group~$C\hspace{-1pt}$.

\vskip 1pt
\end{myromanlist}
Furthermore, 
\begin{myalphlist}

\item
if $G$ is infinite, then the group $C$ in {\rm (ii)} may be 
chosen such that \mbox{$w(C)\hspace{-2pt}= \hspace{-1.5pt}w(G)$};

\item
if $G$ is precompact, then the group $C$ in {\rm (ii)} may be chosen to be 
compact.

\end{myalphlist}
\end{theorem}

Theorem~\ref{LPAq:thm:pcg} is a generalization to locally precompact 
groups of the statement that every compactly generated locally compact 
abelian group is a topological subgroup of a connected locally compact 
abelian group  (cf.~\cite[9.8]{HewRos} and~\cite[23.11]{StroppelLCG}).

In order to prove Theorem~\ref{LPAq:thm:pcg}, we rely on the following 
known, albeit perhaps not sufficiently well-known, result of Morris. We 
are grateful to Kenneth A. Ross for directing us to the cited references.

\begin{ftheorem}[\pcite{Corollary~2}{MorrisVR},
\cite{MorrisVRerr},
\pcite{p.~93, Exercise~1}{Morris}, and \pcite{23.13}{StroppelLCG}]
\label{LPAq:thm:VR}
Every closed subgroup of a compactly generated locally compact abelian 
group is compactly generated.
\end{ftheorem}

\begin{remark}
A recent result of K. H. Hofmann and K.-H. Neeb, which
generalizes Morris's theorem, states that closed (almost) soluble subgroups 
of  (almost) connected locally compact groups are compactly generated
(cf.~\cite{HofNeeb}). However, without such extra assumptions,
statements parallel to  
Theorems~\ref{LPAq:thm:pcg} and~\ref{LPAq:thm:VR} may fail for 
non-abelian groups:

\begin{myalphlist}

\item
The semidirect product   
\mbox{$(\mathbb{Z}/2\mathbb{Z})^\mathbb{Z} \rtimes \mathbb{Z}$}, 
where $\mathbb{Z}$
acts by shifts, is locally compact, but is not pro-Lie 
(cf.~\cite{HofMor3}).
Thus, it cannot be a (closed) topological subgroup of a connected 
locally compact group, as every connected locally compact group is 
pro-Lie (cf.~\cite[Theorem 5']{Yamabe1}).

\item
The commutator subgroup of the (discrete) free group on 
\mbox{$n \hspace{-2pt}> \hspace{-2pt} 1$} generators
is a free group of countable rank
(cf.~\cite[Vol. II, p.~36, Theorem I]{KuroshThG}).

\vspace{6pt}

\end{myalphlist}
\end{remark}

\begin{proof}[Proof of Theorem~\ref{LPAq:thm:pcg}.]
(i) $\Rightarrow$ (ii): By
Lemma~\ref{LPAq:lemma:charpcg}, $\widetilde G$ is compactly generated.
Thus, by replacing $G$ with $\widetilde G$, we may assume that $G$ 
itself is locally compact and compactly generated. (Since
\mbox{$w(G) \hspace{-2pt} = \hspace{-1.5pt} w(\widetilde G)$}, doing so 
does not affect the statement concerning equality of weights.) 
Consequently, 
\mbox{$G\hspace{-1pt}\cong\hspace{-1pt} M \hspace{-2.5pt} \times 
\hspace{-1pt} \mathbb{R}^a \hspace{-2.5pt} \times 
\hspace{-1pt}\mathbb{Z}^c\hspace{-2pt}$},
where $M$ is the maximal compact subgroup of $G\hspace{-1pt}$, and
$a,c \hspace{-2pt}\in\hspace{-2pt} \mathbb{N}$
(cf.~\cite[9.8]{HewRos} and~\cite[23.11]{StroppelLCG}). 
Since $M$ is a subgroup of 
$G$, one has \mbox{$w(M) \hspace{-2pt} \leq \hspace{-2pt} w(G)$}, and
so $M$ is topologically isomorphic to a subgroup of
$\mathbb{T}^{w(G)}\hspace{-1.5pt}$. The group $\mathbb{Z}^c$ is 
a subgroup of $\mathbb{R}^c$.  Therefore,
\mbox{$M \hspace{-2.5pt} \times\hspace{-1pt} \mathbb{R}^a \hspace{-2.5pt} 
\times 
\hspace{-1pt}\mathbb{Z}^c\hspace{0pt}$} is topologically isomorphic to a 
subgroup of the connected group
\mbox{$C\hspace{-2pt} := \hspace{-1pt} \mathbb{T}^{w(G)}\hspace{-2pt} 
\times \hspace{-1pt} \mathbb{R}^{a+c}\hspace{-2pt}$}. 

(a) If $G$ is infinite, then $w(G)$ is infinite, and hence
\mbox{$w(C)\hspace{-2.1pt} =\hspace{-1pt}w(\mathbb{T}^{w(G)})
\hspace{-2pt} = \hspace{-1.1pt} w(G)$}.

(b) If $G$ is precompact, then $\widetilde G$ is compact, and so
\mbox{$a\hspace{-2pt} =\hspace{-2pt}c \hspace{-2pt}= \hspace{-2pt}0$}.
Thus, the group \mbox{$C\hspace{-2pt} := \hspace{-1pt} \mathbb{T}^{w(G)}$} 
is compact, and $G$ is topologically isomorphic to a subgroup of $C$.

(ii) $\Rightarrow$ (i): Suppose that 
\mbox{$G \hspace{-2pt} \cong\hspace{-2pt} S$}, where $S$ is a subgroup of 
$C\hspace{-1.pt}$. Since $C$ is connected, 
it is compactly generated. Thus, by Theorem~\ref{LPAq:thm:VR}, 
$\operatorname{cl}_C S$ is compactly generated too, and so
$\widetilde G$ is compactly generated, because it is topologically 
isomorphic to $\operatorname{cl}_C S$. Hence, by 
Lemma~\ref{LPAq:lemma:charpcg}, $G$ is precompactly generated.
\end{proof}

We are now ready to answer the question in the title of the section.

\begin{theorem} \label{LPAq:thm:main}
Let $A$ be a locally precompact abelian group. The following statements 
are equivalent:

\begin{myromanlist}

\item
$A$ is precompactly generated;

\item
there is a locally pseudocompact abelian group $G$ such that
\mbox{$A \hspace{-2pt}\cong \hspace{-1.7pt} q(G) \hspace{-2.25pt}= 
\hspace{-1.5pt} 
(\widetilde G)_0 \hspace{-2pt} \cap \hspace{-1.5pt} G\hspace{-1pt}$}.

\vskip 1pt

\end{myromanlist}
Furthermore, 

\begin{myalphlist}

\item
if  \mbox{$w(A)\hspace{-2pt} \geq\hspace{-2pt} \omega_1$} and {\rm (i)} holds,
then the group $G$ in {\rm (ii)} may be chosen so that 
\mbox{$w(G)\hspace{-2pt}=\hspace{-1.5pt}w(A)$};

\item
if $A$ is precompact, then the group $G$ in {\rm (ii)} may be chosen to be 
pseudocompact; and

\item
if $A$ is connected, then 
\mbox{$A\hspace{-2pt} \cong \hspace{-2pt} G_0 
\hspace{-1.5pt}= \hspace{-2pt} q(G)$}.

\end{myalphlist}
\end{theorem}

Theorem~\ref{LPAq:thm:main} follows the pattern of a number of known 
``embedding" results, which state that certain (locally) precompact 
groups embed into (locally) pseudocompact groups as a particular (e.g., 
functorial) closed subgroup  (cf.~\cite[2.1]{ComfvMill1}, \cite{Ursul1}, 
\cite[7.6]{ComfvMill2}, \cite{Ursul2}, and \cite[3.6]{Dikconpsc}). 
We are grateful to Dikran Dikranjan 
for suggesting that Theorem~\ref{LPAq:thm:main} might also hold for 
locally  precompact abelian  groups 
(rather than simply for precompact abelian groups, as it appeared in an 
early version of this manuscript), and for drawing 
our attention to the possibility of choosing $G$ so 
that \mbox{$w(G)\hspace{-2pt}=\hspace{-1.5pt}w(A)$}.
In the proof of Theorem~\ref{LPAq:thm:main}, we rely on a well-known 
theorem and a~technical lemma that are presented below.

\begin{ftheorem}[\pcite{9.14 and 24.25}{HewRos}, \pcite{23.27}{StroppelLCG}]
\label{LPAq:thm:condiv}
Every connected locally compact abelian group is divisible.
\end{ftheorem}

In order to distinguish continuous homomorphisms from those that are not 
subject to topological assumptions, we refer to the latter as {\em group 
homomorphisms}.

\begin{lemma} \label{LPAq:lemma:pb}
Let $E$ be a divisible abelian group, and $\lambda$  an infinite cardinal 
such that \mbox{$|E| \hspace{-2pt}\leq\hspace{-2pt} 2^\lambda\hspace{-2pt}$}.
Then for every {\rm (}fixed{\rm )} prime $p$,
there is a group homomorphism
\mbox{$\varphi\colon  \mathbb{Z}_p^{\omega_1 \times \lambda} \rightarrow E$}
such that:

\begin{myalphlist}

\item
$\varphi^{-1}(x)$ is $G_\delta$-dense in  $\mathbb{Z}_p^{\omega_1 \times 
\lambda}$ for every $x \in E$;

\item
for every abelian topological group $C$ and 
group homomorphism
\mbox{$\psi\colon C\rightarrow E$}, the pullback
\begin{align}
\mathbb{Z}_p^{\omega_1 \times \lambda} \times_E C := 
\{(x,c) \in \mathbb{Z}_p^{\omega_1 \times \lambda} \times C \mid 
\varphi(x)=\psi(c)\}
\end{align}
is $G_\delta$-dense in
\mbox{$\mathbb{Z}_p^{\omega_1 \times\lambda} 
\hspace{-3pt}\times\hspace{-2pt} 
C\hspace{-0.5pt}$.}

\end{myalphlist}
\end{lemma}

\begin{proof}
Since the free rank of $\mathbb{Z}_p$ is $2^\omega\hspace{-2.5pt}$,
the free rank of $\mathbb{Z}_p^\lambda$ is $2^\lambda\hspace{-2.5pt}$. 
Thus,
$\mathbb{Z}_p^\lambda$ contains a free abelian subgroup $F$ of rank 
$2^\lambda\hspace{-2.5pt}$. By our assumption, 
\mbox{$|E| \hspace{-2pt}\leq\hspace{-2pt} 2^\lambda\hspace{-2pt}$}, and so 
there exists a surjective group homomorphism
\mbox{$\varphi_0\colon F \rightarrow E$}. One can extend
$\varphi_0$ to a surjective group homomorphism
\mbox{$\varphi_1\colon \mathbb{Z}_p^\lambda \rightarrow E$}, because 
$E$ is divisible. Let
\mbox{$\varphi_2\colon \bigoplus\limits_{\omega_1} \mathbb{Z}_p^\lambda 
\rightarrow E$} denote the group homomorphism
$\bigoplus\limits_{\omega_1} \varphi_1$. Since
$\bigoplus\limits_{\omega_1} \mathbb{Z}_p^\lambda$ is naturally isomorphic 
to a subgroup of
\mbox{$(\mathbb{Z}_p^\lambda)^{\omega_1} \hspace{-2pt} =  \hspace{-2pt}
\mathbb{Z}_p^{\omega_1 \times \lambda}$} and $E$ is divisible,
$\varphi_2$ extends to a group homomorphism
\mbox{$\varphi\colon  \mathbb{Z}_p^{\omega_1 \times \lambda} \rightarrow E$}.
We show that $\varphi$ satisfies the stated properties.

\bigskip

(a) Since translation in $\mathbb{Z}_p^{\omega_1\times\lambda}$ is a 
homeomorphism, 
it suffices to show  that each non-empty $G_\delta$-subset of 
$\mathbb{Z}_p^{\omega_1\times\lambda}$ meets $\varphi^{-1}(0_E)$.
Let $B$ be a non-empty $G_\delta$-subset, and let 
\mbox{$z\hspace{-2pt} \in \hspace{-2pt} B$}. 
There exists
\mbox{$K \hspace{-2pt}\subseteq \hspace{-2pt} \omega_1 \hspace{-2pt}\times
\hspace{-2pt} \lambda$} such that
\begin{align}
N(z,K) := \{y\in  \mathbb{Z}_p^{\omega_1 \times \lambda}
 \mid \forall (\gamma,\delta)\in K,
y_{\gamma,\delta} = z_{\gamma,\delta} \} \subseteq B,
\end{align}
and \mbox{$|K| \hspace{-2pt}\leq \hspace{-2pt}\omega$}.
If \mbox{$K \hspace{-2pt}\cap \hspace{-2pt}
(\{\alpha\} \hspace{-2.25pt}\times \hspace{-1.75pt} \lambda) \hspace{-2pt}
\neq \hspace{-1pt}\emptyset$} for every
\mbox{$\alpha \hspace{-2pt} \in \hspace{-2pt} \omega_1\hspace{-1pt}$},
then \mbox{$|K|\hspace{-2pt} \geq \hspace{-2pt} \omega_1
\hspace{-2pt} >\hspace{-2pt} \omega$}, contrary to the assumption that $K$
is countable. Thus, there is
\mbox{$\alpha_0 \hspace{-2pt} \in \hspace{-2pt} \omega_1\hspace{-1pt}$}
such that \mbox{$K \hspace{-2pt}\cap \hspace{-2pt}
(\{\alpha_0\} \hspace{-2.25pt}\times \hspace{-1.75pt} \lambda)
\hspace{-2pt} =\hspace{-1pt}\emptyset$}.
Since $\varphi_1$ is surjective, there is
\mbox{$w \hspace{-2pt}=\hspace{-2pt} (w_\beta)_{\beta \in \lambda}
 \hspace{-2pt}\in\hspace{-2pt} \mathbb{Z}_p^\lambda$} such
that \mbox{$\varphi_1(w) \hspace{-2.25pt}=\hspace{-1.5pt}\varphi(z)$}.
Let \mbox{$r \hspace{-2pt} \in\hspace{-2pt}
\mathbb{Z}_p^{\omega_1 \times \lambda}$}
denote the element defined
by\begin{align}
r_{\alpha,\beta} =
\begin{cases}
w_\beta & \text{if } \alpha=\alpha_0  \\
0 & \text{otherwise}.
\end{cases}
\end{align}
Since \mbox{$r \hspace{-2pt}\in\hspace{-2pt}  \bigoplus\limits_{\omega_1}
\mathbb{Z}_p^\lambda$}, one has
\mbox{$\varphi(r) \hspace{-2pt}= \hspace{-2pt} \varphi_2(r) 
\hspace{-2pt}= \hspace{-2pt} \varphi_1(w) 
\hspace{-2pt} = \hspace{-2pt} \varphi(z)$},
and therefore
\mbox{$\varphi(z-r) 
\hspace{-2pt}=\hspace{-1.5pt} 0$}.
To conclude, observe that 
\mbox{$z\hspace{-1pt} - \hspace{-1pt}r \hspace{-2pt}\in \hspace{-2pt} N(z,K)$,}
because $z$ and \mbox{$z \hspace{-1pt} - \hspace{-1pt} r$} differ only at
coordinates of the form $(\alpha_0,\beta)$,~and $\alpha_0$ was chosen such 
that \mbox{$K \hspace{-2pt}\cap \hspace{-2pt}
(\{\alpha_0\} \hspace{-2.25pt}\times \hspace{-1.75pt} \lambda)
\hspace{-2pt} =\hspace{-1pt}\emptyset$}. Hence, 
\mbox{$z-r \hspace{-2pt} \in \hspace{-2pt}
N(z,K)\hspace{-2pt}\cap \hspace{-2pt} \varphi^{-1}(0_E) 
\hspace{-2pt} \neq \hspace{-1pt} \emptyset$}, as desired.

(b)  By (a), the set
\mbox{$\varphi^{-1}(\psi(c))\hspace{-2pt} \times \hspace{-2pt} \{c\}$} 
is $G_\delta$-dense in 
\mbox{$\mathbb{Z}_p^{\omega_1 \times \lambda} 
\hspace{-2pt}\times\hspace{-2pt} \{c\}$}
for every \mbox{$c\hspace{-2pt} \in \hspace{-2pt} C\hspace{-1pt}$}. 
Consequently, 
\begin{align}
\mathbb{Z}_p^{\omega_1 \times \lambda} \times_E C = \bigcup\limits_{c\in 
C} 
(\varphi^{-1}(\psi(c)) \times  \{c\})
\end{align}
is $G_\delta$-dense in \mbox{$\mathbb{Z}_p^{\omega_1 \times \lambda} 
\hspace{-2pt}\times\hspace{-2pt} C\hspace{-1pt}$}.
\end{proof}

\begin{proof}[Proof of Theorem~\ref{LPAq:thm:main}.]
(i) $\Rightarrow$ (ii): By Theorem~\ref{LPAq:thm:pcg}, there is a 
connected locally compact abelian group $C$ such that $A$ is topologically 
isomorphic to a subgroup of $C$. Without loss of generality, we may assume 
that $A$ is actually a subgroup of $C$. Put 
\mbox{$E \hspace{-2pt} :=  \hspace{-2pt} C/A$} and 
\mbox{$\lambda \hspace{-2pt} =  \hspace{-2pt} w(C)$}, and let 
\mbox{$\psi\colon C \rightarrow E$} denote the canonical projection. 
\pagebreak[3]
(Unless $A$ is locally 
compact, this quotient is not Hausdorff, but we are interested in $E$ only 
as an abstract group, and ignore its topological properties.)
By Theorem~\ref{LPAq:thm:condiv},  $C$ is divisible, and thus $E$ 
is divisible as well. Since
\mbox{$|E|\hspace{-2pt}\leq \hspace{-2pt} |C| \hspace{-2pt} \leq 
\hspace{-2pt}2^{w(C)}\hspace{-2pt}=\hspace{-2pt} 
2^\lambda\hspace{-2.5pt}$},~the~\mbox{subgroup}
\mbox{$G\hspace{-2pt}:=\hspace{-1pt}
\mathbb{Z}_p^{\omega_1 \times\lambda} \hspace{-3pt}\times_E\hspace{-0pt} C$}
of the group
\mbox{$L\hspace{-2pt} := \hspace{-1pt} \mathbb{Z}_p^{\omega_1 \times\lambda}
\hspace{-3pt}\times\hspace{-2pt}C$} provided by 
Lemma~\ref{LPAq:lemma:pb}(b) is $G_\delta$-dense in $L$. Being a~product 
of a compact and a locally compact group, $L$ is locally compact, and 
consequently 
\mbox{$L\hspace{-2pt} = \hspace{-2pt} \widetilde{G}\hspace{-1pt}$}. 
Therefore, $G$ is locally precompact, and  by 
(the implication (vii) $\Rightarrow$ (i) of)
Theorem~\ref{prel:thm:lcps}, 
$G$ is locally pseudocompact. One has 
\mbox{$L_0 \hspace{-2pt} = \hspace{-2pt}
\{0\} \hspace{-2pt}\times \hspace{-2pt} C$}, because $\mathbb{Z}_p$ is 
zero-dimensional. Hence,
\begin{align}
L_0 \cap G & = \{(x,c)\in \mathbb{Z}_p^{\omega_1 \times \lambda} \times C
\mid\varphi(x)=\psi(c),x=0\} \\
& = \{(0,c)\in\mathbb{Z}_p^{\omega_1\times\lambda}\times C\mid\psi(c) = 0\} = 
\{0\} \times \ker \psi = \{0\} \times A,
\end{align}
where $\varphi$ is the homomorphism constructed in 
Lemma~\ref{LPAq:lemma:pb}.

(a) If \mbox{$w(A)\hspace{-2pt} \geq\hspace{-2pt} \omega_1$}, then 
$A$ is infinite, and therefore by Lemma~\ref{LPAq:thm:pcg}(a), $C$ 
may be chosen such that 
\mbox{$w(C)\hspace{-2pt}=\hspace{-2pt} w(A)$}. Hence,
\begin{align}
w(G) =w(L) = \omega_1 \cdot \lambda \cdot w(C) = w(C) = w(A),
\end{align}
as required.

(b) If $A$ is precompact, then by Lemma~\ref{LPAq:thm:pcg}(b), $C$ may be 
chosen to be compact. Thus, the group $L$ is compact, being a product of 
two compact groups. Therefore, the $G_\delta$-dense subgroup $G$ of $L$ is 
pseudocompact (cf.~\cite{ComfRoss2}). 

(c) If $A$ is connected and $G$ is the group provided by (ii), then $q(G)$ 
is connected, and therefore \mbox{$G_0\hspace{-1.5pt} = \hspace{-2pt} q(G)$}.
Hence, the statement follows by (ii).

(ii) $\Rightarrow$ (i): Since $G$ is locally pseudocompact, 
its completion $\widetilde G$ is locally compact,~and~thus 
$(\widetilde G)_0$ is a~connected locally compact abelian group. 
By our assumption,
$A$ is topologically isomorphic to a subgroup of $(\widetilde G)_0$, 
specifically, to 
\mbox{$(\widetilde G)_0 \hspace{-2pt} \cap \hspace{-1.5pt} G$}. Hence,
by Theorem~\ref{LPAq:thm:pcg}, $A$ is precompactly generated.
\end{proof}

\begin{remark}
We note that when $A$ is metrizable, the indicated equivalence of 
Theorem~\ref{LPAq:thm:main} holds, but the choice of $G$ with
\mbox{$w(G)\hspace{-2pt}=\hspace{-1.5pt}w(A)$} may be impossible.
Indeed, if $A$ is metrizable and precompactly generated, then 
\mbox{$w(A) \hspace{-2pt}=\hspace{-1pt} \omega$}. Consequently,
if $G$ is a locally pseudocompact group such that
\mbox{$w(G)\hspace{-2pt}=\hspace{-1.5pt}w(A)$}, then $G$ is 
locally compact. Thus, $q(G)$ is locally compact (being a closed subgroup 
of $G$), and
by Theorem~\ref{conn:thm:LC}, 
$q(G)$ is connected. Hence, $A$ can be topologically isomorphic to $q(G)$ 
only if $A$ itself is connected and locally compact; in that case, one 
can take \mbox{$G\hspace{-2pt} =\hspace{-2pt}A$}. In particular, for
\mbox{$A\hspace{-2pt}: = \hspace{-1.5pt}\mathbb{Q}$}, no locally 
pseudocompact $G$ can satisfy both (ii) of Theorem~\ref{LPAq:thm:main} and
\mbox{$w(G)\hspace{-2pt}=\hspace{-1.5pt}w(A)$}.
\end{remark}

\bigskip

\section*{Acknowledgements}

The authors wish to thank Dikran Dikranjan, Kenneth A. Ross,
and Javier Trigos-Arrieta for 
valuable electronic correspondence, and James D. Reid and Grant Woods for 
valuable discussions. The authors are grateful to Karen Kipper for her 
kind help in proofreading this paper for grammar and punctuation.
The authors appreciate the wealth of constructive comments of the 
anonymous referees that led to an improved presentation of this paper.

{\footnotesize

\bibliography{notes,notes2,notes3}

\def\cprime{$'$} \def\cprime{$'$}
\begin{thebibliography}{10}

\bibitem{Arh}
A.~V. Arhangel{\cprime}ski\u{\i}.
\newblock Cardinal invariants of topological groups. {E}mbeddings and
  condensations.
\newblock {\em Dokl. Akad. Nauk SSSR}, 247(4):779--782, 1979.

\bibitem{ArhTka}
A.~V. Arhangel{\cprime}ski\u{\i} and M.~Tkachenko.
\newblock {\em Topological groups and related structures}.
\newblock Atlantis Studies in Mathematics. World Scientific, 2008.

\bibitem{Blair}
R.~L. Blair.
\newblock Spaces in which special sets are {$z$}-embedded.
\newblock {\em Canad. J. Math.}, 28(4):673--690, 1976.

\bibitem{BlaHag2}
R.~L. Blair and A.~W. Hager.
\newblock Notes on the {H}ewitt realcompactification of a product.
\newblock {\em General Topology and Appl.}, 5:1--8, 1975.

\bibitem{ComfHernTrig}
W.~W. Comfort, S.~Hern{\'a}ndez, and F.~J. Trigos-Arrieta.
\newblock Relating a locally compact abelian group to its {B}ohr
  compactification.
\newblock {\em Adv. Math.}, 120(2):322--344, 1996.

\bibitem{ComfvMill1}
W.~W. Comfort and J.~van Mill.
\newblock On the existence of free topological groups.
\newblock {\em Topology Appl.}, 29(3):245--265, 1988.

\bibitem{ComfvMill2}
W.~W. Comfort and J.~van Mill.
\newblock Concerning connected, pseudocompact abelian groups.
\newblock {\em Topology Appl.}, 33(1):21--45, 1989.

\bibitem{ComfNegB2}
W.~W. Comfort and S.~Negrepontis.
\newblock {\em Continuous pseudometrics}.
\newblock Marcel Dekker Inc., New York, 1975.
\newblock Lecture Notes in Pure and Applied Mathematics, Vol. 14.

\bibitem{ComfRoss2}
W.~W. Comfort and K.~A. Ross.
\newblock Pseudocompactness and uniform continuity in topological groups.
\newblock {\em Pacific J. Math.}, 16:483--496, 1966.

\bibitem{ComfTrig}
W.~W. Comfort and F.~J. Trigos-Arrieta.
\newblock Locally pseudocompact topological groups.
\newblock {\em Topology Appl.}, 62(3):263--280, 1995.

\bibitem{dieu39}
J.~Dieudonn{\'e}.
\newblock Sur les espaces uniformes complets.
\newblock {\em Ann. \'Ecole Norm.}, 56:277--291, 1939.

\bibitem{Dikconpsc}
D.~Dikranjan.
\newblock Connectedness and disconnectedness in pseudocompact groups.
\newblock {\em Rend. Accad. Naz. Sci. XL Mem. Mat. (5)}, 16:211--221, 1992.

\bibitem{Dikdimpsc}
D.~Dikranjan.
\newblock Dimension and connectedness in pseudo-compact groups.
\newblock {\em C. R. Acad. Sci. Paris S\'er. I Math.}, 316(4):309--314, 1993.

\bibitem{DikPS0dim}
D.~Dikranjan.
\newblock Zero-dimensionality of some pseudocompact groups.
\newblock {\em Proc. Amer. Math. Soc.}, 120(4):1299--1308, 1994.

\bibitem{Dikconcomp}
D.~Dikranjan.
\newblock Compactness and connectedness in topological groups.
\newblock {\em Topology Appl.}, 84(1-3):227--252, 1998.

\bibitem{DikCOTA}
D.~Dikranjan.
\newblock Closure operators in topological groups related to von {N}eumann's
  kernel.
\newblock {\em Topology Appl.}, 153(11):1930--1955, 2006.

\bibitem{DikCLP}
D.~Dikranjan.
\newblock C{LP}-compactness for topological spaces and groups.
\newblock {\em Topology Appl.}, 154(7):1321--1340, 2007.

\bibitem{DikGL3}
D.~Dikranjan and G.~Luk{\'a}cs.
\newblock On zero-dimensionality and the connected component of locally
  pseudocompact groups.
\newblock {\em Submitted}.
\newblock ArXiv: 0909.1390.

\bibitem{DikProSto}
D.~Dikranjan, I.~R. Prodanov, and L.~N. Stoyanov.
\newblock {\em Topological groups}, volume 130 of {\em Monographs and Textbooks
  in Pure and Applied Mathematics}.
\newblock Marcel Dekker Inc., New York, 1990.
\newblock Characters, dualities and minimal group topologies.

\bibitem{DikTkaWCFTP}
D.~Dikranjan and M.~Tkachenko.
\newblock Weakly complete free topological groups.
\newblock {\em Topology Appl.}, 112(3):259--287, 2001.

\bibitem{Engel6}
R.~Engelking.
\newblock {\em General topology}, volume~6 of {\em Sigma Series in Pure
  Mathematics}.
\newblock Heldermann Verlag, Berlin, second edition, 1989.
\newblock Translated from the Polish by the author.

\bibitem{GilJer}
L.~Gillman and M.~Jerison.
\newblock {\em Rings of continuous functions}.
\newblock The University Series in Higher Mathematics. D. Van Nostrand Co.,
  Inc., Princeton, N.J.-Toronto-London-New York, 1960.

\bibitem{Glick3}
I.~Glicksberg.
\newblock The representation of functionals by integrals.
\newblock {\em Duke Math. J.}, 19:253--261, 1952.

\bibitem{Graev}
M.~I. Graev.
\newblock Theory of topological groups {I}. {N}orms and metrics on groups.
  {C}omplete groups. {F}ree topological groups.
\newblock {\em Uspehi Matem. Nauk (N.S.)}, 5(2(36)):3--56, 1950.

\bibitem{GrantOMT}
D.~L. Grant.
\newblock Topological groups which satisfy an open mapping theorem.
\newblock {\em Pacific J. Math.}, 68(2):411--423, 1977.

\bibitem{Guran}
I.~I. Guran.
\newblock On topological groups close to being {L}indel\"of.
\newblock {\em Soviet Math. Dokl.}, 23:173--175, 1981.

\bibitem{Hager1}
A.~W. Hager.
\newblock On inverse-closed subalgebras of {$C(X)$}.
\newblock {\em Proc. London Math. Soc. (3)}, 19:233--257, 1969.

\bibitem{HenJohn}
M.~Henriksen and D.~G. Johnson.
\newblock On the structure of a class of archimedean lattice-ordered algebras.
\newblock {\em Fund. Math.}, 50:73--94, 1961/1962.

\bibitem{HernMaca}
S.~Hern{\'a}ndez and S.~Macario.
\newblock Dual properties in totally bounded abelian groups.
\newblock {\em Arch. Math. (Basel)}, 80(3):271--283, 2003.

\bibitem{HewRVCF}
E.~Hewitt.
\newblock Rings of real-valued continuous functions. {I}.
\newblock {\em Trans. Amer. Math. Soc.}, 64:45--99, 1948.

\bibitem{HewRos}
E.~Hewitt and K.~A. Ross.
\newblock {\em Abstract harmonic analysis}.
\newblock Academic Press Inc., Publishers, New York, 1963.

\bibitem{HofMor}
K.~H. Hofmann and S.~A. Morris.
\newblock {\em The structure of compact groups}, volume~25 of {\em de Gruyter
  Studies in Mathematics}.
\newblock Walter de Gruyter \& Co., Berlin, 1998.
\newblock A primer for the student---a handbook for the expert.

\bibitem{HofMor3}
K.~H. Hofmann and S.~A. Morris.
\newblock Projective limits of finite-dimensional {L}ie groups.
\newblock {\em Proc. London Math. Soc. (3)}, 87(3):647--676, 2003.

\bibitem{HofMor2}
K.~H. Hofmann and S.~A. Morris.
\newblock {\em The {L}ie theory of connected pro-{L}ie groups}, volume~2 of
  {\em EMS Tracts in Mathematics}.
\newblock European Mathematical Society (EMS), Z\"urich, 2007.
\newblock A structure theory for pro-Lie algebras, pro-Lie groups, and
  connected locally compact groups.

\bibitem{HofNeeb}
K.~H. Hofmann and K.-H. Neeb.
\newblock The compact generation of closed subgroups of locally compact groups.
\newblock {\em J. Group Theory}, 12(4):555--559, 2009.

\bibitem{Isbell5}
J.~R. Isbell.
\newblock Algebras of uniformly continuous functions.
\newblock {\em Ann. of Math. (2)}, 68:96--125, 1958.

\bibitem{Isb}
J.~R. Isbell.
\newblock {\em Uniform spaces}.
\newblock American Mathematical Society, Providence, R.I., 1964.

\bibitem{Kac}
G.~I. Kac.
\newblock Isomorphic mapping of topological groups into a direct product of
  groups satisfying the first countability axiom.
\newblock {\em Uspehi Matem. Nauk (N.S.)}, 8(6(58)):107--113, 1953.

\bibitem{Kelley}
J.~L. Kelley.
\newblock {\em General topology}.
\newblock D. Van Nostrand Company, Inc., Toronto-New York-London, 1955.

\bibitem{KunenST}
K.~Kunen.
\newblock {\em Set theory}, volume 102 of {\em Studies in Logic and the
  Foundations of Mathematics}.
\newblock North-Holland Publishing Co., Amsterdam, 1983.
\newblock An introduction to independence proofs, Reprint of the 1980 original.

\bibitem{KuroshThG}
A.~G. Kurosh.
\newblock {\em The theory of groups}.
\newblock Chelsea Publishing Co., New York, 1960.
\newblock Translated from the Russian and edited by K. A. Hirsch. 2nd English
  ed. 2 volumes.

\bibitem{GLCLTG}
G.~Luk{\'a}cs.
\newblock {\em Compact-like topological groups}, volume~31 of {\em Research and
  exposition in Mathematics}.
\newblock Heldermann Verlag, Berlin, 2009.

\bibitem{MaRaWo}
J.~Mack, M.~Rayburn, and G.~Woods.
\newblock Local topological properties and one point extensions.
\newblock {\em Canad. J. Math.}, 24:338--348, 1972.

\bibitem{MackeyMeas}
G.~W. Mackey.
\newblock Equivalence of a problem in measure theory to a problem in the theory
  of vector lattices.
\newblock {\em Bull. Amer. Math. Soc.}, 50:719--722, 1944.

\bibitem{MorrisVR}
S.~A. Morris.
\newblock Locally compact abelian groups and the variety of topological groups
  generated by the reals.
\newblock {\em Proc. Amer. Math. Soc.}, 34:290--292, 1972.

\bibitem{MorrisVRerr}
S.~A. Morris.
\newblock Erratum to: ``{L}ocally compact abelian groups and the variety of
  topological groups generated by the reals'' ({P}roc. {A}mer. {M}ath. {S}oc. {
  34} (1972), 290--292).
\newblock {\em Proc. Amer. Math. Soc.}, 51:503, 1975.

\bibitem{Morris}
S.~A. Morris.
\newblock {\em Pontryagin duality and the structure of locally compact abelian
  groups}.
\newblock Cambridge University Press, Cambridge, 1977.
\newblock London Mathematical Society Lecture Note Series, No. 29.

\bibitem{Rai}
D.~Ra\v{\i}kov.
\newblock On the completion of topological groups.
\newblock {\em Bull. Acad. Sci. URSS. S\'er. Math. [Izvestia Akad. Nauk SSSR]},
  10:513--528, 1946.

\bibitem{RoeDie}
W.~Roelcke and S.~Dierolf.
\newblock {\em Uniform structures on topological groups and their quotients}.
\newblock McGraw-Hill International Book Co., New York, 1981.
\newblock Advanced Book Program.

\bibitem{RossStro}
K.~A. Ross and K.~Stromberg.
\newblock Baire sets and {B}aire measures.
\newblock {\em Ark. Mat.}, 6:151--160 (1965), 1965.

\bibitem{Scepin1}
E.~V. {\v{S}}{\v{c}}epin.
\newblock Topological products, groups, and a new class of spaces that are more
  general than metric spaces.
\newblock {\em Dokl. Akad. Nauk SSSR}, 226(3):527--529, 1976.

\bibitem{Scepin2}
E.~V. {\v{S}}{\v{c}}epin.
\newblock {$\kappa $}-metrizable spaces.
\newblock {\em Izv. Akad. Nauk SSSR Ser. Mat.}, 43(2):442--478, 1979.
\newblock English translation: Math. USSR-Izv. 14 (1980), no. 2, 407--440.

\bibitem{Shirota}
T.~Shirota.
\newblock A class of topological spaces.
\newblock {\em Osaka Math. J.}, 4:23--40, 1952.

\bibitem{StroppelLCG}
M.~Stroppel.
\newblock {\em Locally compact groups}.
\newblock EMS Textbooks in Mathematics. European Mathematical Society (EMS),
  Z\"urich, 2006.

\bibitem{SulleyOMT}
L.~J. Sulley.
\newblock A note on {$B$}- and {$B\sb{r}$}-complete topological {A}belian
  groups.
\newblock {\em Proc. Cambridge Philos. Soc.}, 66:275--279, 1969.

\bibitem{Tarski}
A.~Tarski.
\newblock Drei \"{U}berdeckungss\"atze der allgemeinen {M}engenlehre.
\newblock {\em Fund. Math.}, 30:132--155, 1938.

\bibitem{TkaCemb}
M.~Tkachenko.
\newblock The notion of {${\rm o}$}-tightness and {$C$}-embedded subspaces of
  products.
\newblock {\em Topology Appl.}, 15(1):93--98, 1983.

\bibitem{TkaIntro}
M.~Tkachenko.
\newblock Introduction to topological groups.
\newblock {\em Topology Appl.}, 86(3):179--231, 1998.

\bibitem{Ulam}
S.~Ulam.
\newblock Zur {M}asstheorie in der allgemeinen {M}engenlehre.
\newblock {\em Fund. Math.}, 16:140--150, 1930.

\bibitem{UrsulExpl}
M.~I. Ursul.
\newblock An example of a plane group whose quasicomponent does not coincide
  with its component.
\newblock {\em Mat. Zametki}, 38(4):517--522, 634, 1985.

\bibitem{Ursul1}
M.~I. Ursul.
\newblock Embedding of locally precompact groups into locally pseudocompact
  groups.
\newblock {\em Izv. Akad. Nauk Moldav. SSR Ser. Fiz.-Tekhn. Mat. Nauk},
  (3):54--56, 77, 1989.

\bibitem{Ursul2}
M.~I. Ursul.
\newblock Embedding of precompact groups into pseudocompact groups.
\newblock {\em Izv. Akad. Nauk Moldav. SSR Mat.}, (1):88--89, 93, 1991.

\bibitem{We3}
A.~Weil.
\newblock {\em Sur les {E}spaces \`a {S}tructure {U}niforme et sur la
  {T}opologie {G}\'en\'erale}.
\newblock Publ. Math. Univ. Strasbourg. Hermann et Cie., Paris, 1937.

\bibitem{Yamabe1}
H.~Yamabe.
\newblock A generalization of a theorem of {G}leason.
\newblock {\em Ann. of Math. (2)}, 58:351--365, 1953.

\end{thebibliography}
}

\begin{samepage}

\bigskip
\noindent
\begin{tabular}{l @{\hspace{1.9cm}} l}
Department of Mathematics & Department of Mathematics\\
Wesleyan University & University of Manitoba\\
Middletown, CT 06459 & Winnipeg, Manitoba, R3T 2N2 \\
USA & Canada \\ & \\
\em e-mail: wcomfort@wesleyan.edu  &
\em e-mail: lukacs@cc.umanitoba.ca
\end{tabular}

\end{samepage}

\end{document}